\pdfoutput=1                      
\documentclass[10 pt,a4paper]{article}

\usepackage[accsupp]{axessibility}    
\usepackage{latexsym,amsfonts,amsmath,amsthm,amssymb, graphicx, mathrsfs, fancyhdr, makeidx, amscd,  color, mathdesign, fontenc, stix, hyperref, enumitem, soul}
\usepackage{tocloft}
\usepackage[all]{xy}
\newtheorem{theorem}{Theorem}[section]
\newtheorem{definition}[theorem]{Definition}
\newtheorem{proposition}[theorem]{Proposition}
\newtheorem{lemma}[theorem]{Lemma}
\newtheorem{corollary}[theorem]{Corollary}
\newtheorem{conjecture}[theorem]{Conjecture}
\newtheorem{question}[theorem]{Question}

\numberwithin{equation}{section}
\numberwithin{theorem}{section}

\theoremstyle{definition}
\newtheorem{remark}[theorem]{Remark}
\newtheorem{example}[theorem]{Example}

\DeclareMathOperator{\Lip}{Lip}
\newcommand{\D}{\nabla}
\newcommand{\p}{\partial}
\renewcommand{\H}{\mathcal{H}}
\renewcommand{\P}[1]{{P\big(#1\big)}}
\newcommand{\Ps}[1]{{P(#1)}}
\newcommand{\RR}{\mathbb{R}}
\newcommand{\uu}{\underline{u}}
\newcommand{\ou}{\overline{u}}
\renewcommand{\bar}{\overline}

\begin{document}

%\pagestyle{fancy}
%% i comandi seguenti impediscono la scrittura in maiuscolo
%% dei nomi dei capitoli e dei paragrafi nelle intestazioni
%\renewcommand{\chaptermark}[1]{\markboth{#1}{}}
%\renewcommand{\sectionmark}[1]{\markright{\thesection\ #1}}
%\fancyhf{} % rimuove l'attuale contenuto dell'intestazione
%% e del pi\`e di pagina
%\fancyhead[LE,RO]{\bfseries\thepage}
%\fancyhead[LO]{\bfseries\rightmark}
%\fancyhead[RE]{\bfseries\leftmark}
%\renewcommand{\headrulewidth}{0.5pt}
%\renewcommand{\footrulewidth}{0pt}
%\addtolength{\headheight}{0.5pt} % riserva spazio per la linea
%\fancypagestyle{plain}{%
%\fancyhead{} % ignora, nello stile plain, le intestazioni
%\renewcommand{\headrulewidth}{0pt} % e la linea
%}

\newcommand{\riem}{(M^m, \langle \, , \, \rangle)}
\newcommand{\Hess}{\mathrm{Hess}\, }
\newcommand{\hess}{\mathrm{hess}\, }
\newcommand{\cut}{\mathrm{cut}}
\newcommand{\ind}{\mathrm{ind}}
\newcommand{\ess}{\mathrm{ess}}
\newcommand{\longra}{\longrightarrow}
\newcommand{\eps}{\varepsilon}
\newcommand{\ra}{\rightarrow}
\newcommand{\vol}{\mathrm{vol}}
\newcommand{\di}{\mathrm{d}}
\newcommand{\R}{\mathbb R}
\newcommand{\C}{\mathbb C}
\newcommand{\Z}{\mathbb Z}
\newcommand{\N}{\mathbb N}
\newcommand{\HH}{\mathbb H}
\newcommand{\esse}{\mathbb S}
\newcommand{\bull}{\rule{2.5mm}{2.5mm}\vskip 0.5 truecm}
\newcommand{\binomio}[2]{\genfrac{}{}{0pt}{}{#1}{#2}} %identico ad \atop
\newcommand{\metric}{\langle \, , \, \rangle}
\newcommand{\lip}{\mathrm{Lip}}
\newcommand{\loc}{\mathrm{loc}}
\newcommand{\diver}{\mathrm{div}}
\newcommand{\disp}{\displaystyle}
\newcommand{\rad}{\mathrm{rad}}
\newcommand{\mmetric}{\langle\langle \, , \, \rangle\rangle}
\newcommand{\sn}{\mathrm{sn}}
\newcommand{\cn}{\mathrm{cn}}
\newcommand{\ink}{\mathrm{in}}
\newcommand{\hol}{\mathrm{H\ddot{o}l}}
\newcommand{\capac}{\mathrm{cap}}
\newcommand{\bmo}{\{b <0\}}
\newcommand{\bmuo}{\{b \le 0\}}
\newcommand{\Fk}{\mathcal{F}_k}
\newcommand{\dist}{\mathrm{dist}}
\newcommand{\gr}{\mathcal{G}}
\newcommand{\grh}{\mathcal{G}^{h}}
\newcommand{\sgrh}{\mathscr{G}^h}
\newcommand{\sgr}{\mathscr{G}}
\newcommand{\Ricc}{\mathrm{Ric}}
\newcommand{\Sect}{\mathrm{Sec}}
\newcommand{\roy}{\mathsf{Ro}}
\newcommand{\roypsi}{\mathsf{Ro}_\psi}
\newcommand{\LL}{\mathcal{L}}
\newcommand{\Do}{\mathscr{D}_o}
\newcommand{\Dom}{\mathscr{D}_\Omega}
\newcommand{\BBB}{\mathscr{B}}
\newcommand{\Ric}{\mathrm{Ric}}

\newcommand{\vp}{\varphi}
\renewcommand{\div}[1]{{\mathop{\mathrm div}}\left(#1\right)}
\newcommand{\divphi}[1]{{\mathop{\mathrm div}}\bigl(\vert \nabla #1
\vert^{-1} \varphi(\vert \nabla #1 \vert)\nabla #1   \bigr)}
\newcommand{\nablaphi}[1]{\vert \nabla #1\vert^{-1}
\varphi(\vert \nabla #1 \vert)\nabla #1}
\newcommand{\modnabla}[1]{\vert \nabla #1\vert }
\newcommand{\modnablaphi}[1]{\varphi\bigl(\vert \nabla #1 \vert\bigr)
\vert \nabla #1\vert }
\newcommand{\ds}{\displaystyle}
\newcommand{\cL}{\mathcal{L}}
\newcommand{\essem}{\mathds{S}^m}
\newcommand{\erre}{\mathds{R}}
\newcommand{\errem}{\mathds{R}^m}
\newcommand{\enne}{\mathds{N}}
\newcommand{\acca}{\mathds{H}}

\newcommand{\cvett}{\Gamma(TM)}
\newcommand{\cinf}{C^{\infty}(M)}
\newcommand{\sptg}[1]{T_{#1}M}
\newcommand{\partder}[1]{\frac{\partial}{\partial {#1}}}
\newcommand{\partderf}[2]{\frac{\partial {#1}}{\partial {#2}}}
\newcommand{\ctloc}{(\mathcal{U}, \varphi)}
\newcommand{\fcoord}{x^1, \ldots, x^n}
\newcommand{\ddk}[2]{\delta_{#2}^{#1}}
\newcommand{\christ}{\Gamma_{ij}^k}
\newcommand{\ricc}{\operatorname{Ricc}}
\newcommand{\supp}{\operatorname{supp}}
\newcommand{\sgn}{\operatorname{sgn}}
\newcommand{\rg}{\operatorname{rg}}
\newcommand{\inv}[1]{{#1}^{-1}}
\newcommand{\id}{\operatorname{id}}
\newcommand{\jacobi}[3]{\sq{\sq{#1,#2},#3}+\sq{\sq{#2,#3},#1}+\sq{\sq{#3,#1},#2}=0}
\newcommand{\lie}{\mathfrak{g}}
\newcommand{\rp}{\erre\mathds{P}}
\newcommand{\II}{\operatorname{II}}
\newcommand{\gradh}[1]{\nabla_{H^m}{#1}}
\newcommand{\absh}[1]{{\left|#1\right|_{H^m}}}
\newcommand{\mob}{\mathrm{M\ddot{o}b}}
\newcommand{\mab}{\mathfrak{m\ddot{o}b}}
\newcommand{\foc}{\mathrm{foc}}
\newcommand{\F}{\mathcal{F}}
\newcommand{\Cf}{\mathcal{C}_f}
\newcommand{\cutf}{\mathrm{cut}_{f}}
\newcommand{\Cn}{\mathcal{C}_n}
\newcommand{\cutn}{\mathrm{cut}_{n}}
\newcommand{\Ca}{\mathcal{C}_a}
\newcommand{\cuta}{\mathrm{cut}_{a}}
\newcommand{\cutc}{\mathrm{cut}_c}
\newcommand{\cutcf}{\mathrm{cut}_{cf}}
\newcommand{\rk}{\mathrm{rk}}
\newcommand{\crit}{\mathrm{crit}}
\newcommand{\diam}{\mathrm{diam}}
\newcommand{\haus}{\mathscr{H}}
\newcommand{\po}{\mathrm{po}}
\newcommand{\grg}{\mathcal{G}^{(g)}}
\newcommand{\grj}{\mathcal{G}^{(j)}}
\newcommand{\Sph}{\mathbb{S}}
\newcommand{\BB}{\mathbb{B}}
\newcommand{\EE}{\mathbb{E}}
\newcommand{\Gia}{G^{(a)}}
\newcommand{\Gib}{G^{(b)}}
\newcommand{\tr}{\mathrm{Tr}}
\newcommand{\dou}{\mathrm{(VD)}}
\newcommand{\dous}{\mathrm{(VD)}}
\newcommand{\neuqp}{(\mathrm{NP}_{q,p})}
\newcommand{\neup}{(\mathrm{NP}_{p,p})}
\newcommand{\neuq}{(\mathrm{NP}_{q,q})}
\newcommand{\neuuno}{(\mathrm{NP}_{1,1})}
\newcommand{\neuunop}{(\mathrm{NP}_{1,p})}

\newcommand{\uniho}{\mathrm{(UO)}}
\newcommand{\vcom}{\mathrm{(VC)}}
\newcommand{\pvcom}{\mathrm{(PVC)}}
\newcommand{\tcr}{\textcolor{red}}
\newcommand{\tcb}{\textcolor{blue}}
\newcommand{\tcg}{\textcolor{green}}
\newcommand{\ol}{\overline}
\renewcommand{\ul}{\underline}
\newcommand{\Hp}{\mathscr{H}_p}
\newcommand{\vut}{\{\varrho=t\}}
\newcommand{\vmt}{\{\varrho < t\}}
\newcommand{\vutp}{\{\varrho_p=t\}}
\newcommand{\vmtp}{\{\varrho_p \le t\}}
\newcommand{\vus}{\{\varrho=s\}}
\newcommand{\vms}{\{\varrho \le s\}}
\newcommand{\Ar}{\mathscr{A}}
\newcommand{\VV}{\mathscr{V}}
\newcommand{\metricN}{( \, , \, )}
\newcommand{\Po}{\mathscr{P}}
\newcommand{\Ha}{\mathscr{H}}
\newcommand{\So}{\mathscr{S}}
\newcommand{\Dou}{\mathscr{D}}
\newcommand{\Vo}{\mathscr{V}}
\newcommand{\RD}{\mathscr{R}}
\newcommand{\DM}{\mathcal{DM}}

\newcommand{\blue}[1]{\textcolor{blue}{#1}}
\newcommand{\w}{\mathscr{w}}

\newcommand{\Hm}[1]{\leavevmode{\marginpar{\tiny%
			$\hbox to 0mm{\hspace*{-0.5mm}$\leftarrow$\hss}%
			\vcenter{\vrule depth 0.1mm height 0.1mm width \the\marginparwidth}%
			\hbox to 0mm{\hss$\rightarrow$\hspace*{-0.5mm}}$\\\relax\raggedright
			#1}}}

\author{Luca Benatti \and Luciano Mari \and Marco Rigoli \and Alberto G. Setti \and Kai Xu}

\title{\textbf{Proper solutions of the $1/H$-flow and the Green kernel of the $p$-Laplacian}}

% \title{\textbf{Decay estimates for the Green kernel of the $p$-Laplacian and the $1/H$-flow}}
\date{}
\maketitle

% \scriptsize \begin{center} Dipartimento di Matematica, Universit\`a degl Studi di Torino,\\
% Via Carlo Alberto 10, 10123 Torino (Italy)\\
% E-mail: luciano.mari@unito.it
% \end{center}

% \scriptsize \begin{center} Dipartimento di Matematica,
% Universit\`a
% degli Studi di Milano,\\
% Via Saldini 50, I-20133 Milano (Italy)\\
% E-mail: marco.rigoli55@gmail.com
% \end{center}

% \scriptsize \begin{center} Dipartimento di Scienza e Alta Tecnologia,
% Universit\`a degli Studi dell'Insubria,\\
% Via Valleggio 11, I-22100 Como (Italy)\\
% E-mail: alberto.setti@uninsubria.it
% \end{center}

\normalsize

\begin{abstract}
    We show existence and optimal growth estimate for the weak inverse mean curvature flow issuing from a point, on manifolds with certain curvature and isoperimetric conditions. These theorems imply analogous ones for the flow issuing from relatively compact sets. Some of the results are obtained by proving new decay estimates for the Green kernel of the $p$-Laplacian which fix a gap in the literature. Additionally, we address the convergence of renormalized $p$-capacitary potentials to the inverse mean curvature flow with outer obstacle.
\end{abstract}

\tableofcontents

\section{Introduction}

Since its introduction to prove the Penrose inequality in General Relativity \cite{huiskenilmanen}, the weak formulation of the Inverse Mean Curvature Flow (IMCF) has shown to be a powerful tool in investigating the geometry of complete, noncompact Riemannian manifolds, see for instance \cite{bfm,bray_neves,cho_lai_xu,fogamazzie,huiskenkorber,mccormick,Shi_2016}. Therefore, constructing the flow and estimating its growth in a sharp way under mild geometric requests on the underlying manifold is desirable. In \cite{huiskenilmanen}, G. Huisken and T. Ilmanen were able to produce a weak IMCF in any complete manifold $M$ whenever the latter supports a proper subsolution. While finding such subsolution is easy in the setting of asymptotically flat manifolds, considered in \cite{huiskenilmanen}, it is highly nontrivial in manifolds whose geometry at infinity is controlled in a milder way. The issue of existence of a proper flow was tackled in some works appearing in the past 15 years, \cite{kotschwarni,ajm,Xu_2023_proper}, which will be commented later. In this respect, the purpose of the present paper is to produce general existence results with sharp growth estimates, improving on the aforementioned papers and correcting a gap in the literature.

\subsection{Main theorems}

% The purpose of this paper is to produce weak solutions to the Inverse Mean Curvature Flow (IMCF) with sharp growth estimates, on complete manifolds whose geometry at infinity is mildly controlled.

Henceforth, all manifolds $M$ are assumed to be smooth, complete, connected, noncompact and of dimension $n \ge 2$. We fix a basepoint $o\in M$, and denote $r(x)=d(x,o)$. We adopt the following notations: for a function $u$ we set $E_t(u)=\{u<t\}$,  for short $E_t$ when there is no ambiguity. We denote by $\Ps{E;K}$ the perimeter of a set $E$ in $K$, and write $\Ps{E}=\Ps{E;M}$.

According to \cite{huiskenilmanen}, we say that $u\in\Lip_{\loc}(\Omega)$ is a weak IMCF in a domain $\Omega$ if for each $t\in\RR$ and $K\Subset\Omega$ we have
\[
J_u^K(E_t)\leq J_u^K(F),\qquad\forall\,F\text{ such that }E\Delta F\Subset K,
\]
where the energy $J_u^K$ is defined as
\[
J_u^K(E):=\Ps{E;K}-\int_{E\cap K}|\D u|.
\]
For detailed introductions of the weak IMCF, see \cite{huiskenilmanen} and also \cite[Section 2]{Xu_2023_proper}. An important role in our paper will be played by inverse mean curvature flows issuing from a given point $o \in M$, that is, an IMCF $u$ defined on $M\backslash\{o\}$, such that
    \[
    u(x)-(n-1)\log r(x)\to0\qquad\text{as}\ \ x\to o.
    \] 
Following Huisken, we refer to them as \textit{IMCF cores}. An IMCF core $u$ is called \textit{proper} if $\{u<t\}\Subset M$ for all $t\in\RR$. By the maximum principle in \cite[Theorem 2.2]{huiskenilmanen}, there exists at most one proper IMCF core with pole $o$. 

The following are the main theorems of this work:

\begin{theorem}\label{thm-intro:main_ric_rd}
    Suppose $M$ satisfies $\Ric\geq0$ and a reverse volume doubling condition
    \begin{equation}\label{eq-intro:rd}
        \frac{|B_t(o)|}{|B_s(o)|} \ge C_{\RD} \left( \frac{t}{s}\right)^b,\qquad\forall\,t\geq s>0. \tag{$\text{RD}_b$}
    \end{equation}
    for some $b\in(1,n]$ and $C_{\RD}>0$. Then there exists a unique proper IMCF core $u$ on $M$ with pole $o$, satisfying the growth estimate
    \begin{align}
        & u(x)\geq(n-1)\log r(x)-C\big(n,b,C_{\RD},|B_1(o)|\big),\qquad \forall\,x\in B_1(o), \label{eq-intro:growth_ric_rd}\\
        & u(x)\geq(b-1)\log r(x)-C\big(n,b,C_{\RD},|B_1(o)|\big),\qquad \forall\,x\in M\backslash B_1(o), \label{eq-intro:growth_ric_rd_2}
    \end{align}
    and the global gradient estimate
    \begin{equation}\label{eq_gradu_ricge0}
        |\D u(x)|\leq(n-1)e^{-\frac{u}{n-1}},\qquad\forall\,x\in M\backslash\{o\}.
    \end{equation}
\end{theorem}

\begin{theorem}\label{thm-intro:main_euc_isop}
    Suppose that the Euclidean isoperimetric inequality 
    \begin{equation}\label{eq-intro:euc_isop}
        \Ps{E}\geq c_I|E|^{\frac{n-1}n},\qquad\forall\ E\Subset M
    \end{equation}
holds in $M$ for some constant $c_I>0$. 
Then there exists a unique proper IMCF core $u$ on $M$ with pole $o$, satisfying  the growth estimate
    \begin{equation}\label{eq-intro:growth_euc_isop}
        u(x)\geq(n-1)\log r(x)-C(n,c_I),\qquad \forall\,x\in M\backslash\{o\}.
    \end{equation}
\end{theorem}

\begin{remark}
Theorem \ref{thm-intro:main_ric_rd} also holds if condition $\Ric \ge 0$ is replaced by a more general Ricci lower bound, provided that a global volume doubling and weak $(1,1)$-Neumann Poincar\'e inequality hold on $M$. In this case, the gradient estimate \ref{eq_gradu_ricge0} changes according to the lower bound on $\Ric$.
%(with the appropriate dependencies on the %constant $C$). 
See the comments on \cite[Theorem 4.6]{ajm} in Subsection \ref{subsec:ajm} for more details.
\end{remark}

As was observed in \cite{ajm}, IMCF cores are useful in solving the usual initial value problem of IMCF. Let $E_0\Subset M$ be a $C^{1,1}$ domain. We say that
\begin{itemize}[topsep=1pt, itemsep=-0.5ex]
    \item $u\in\Lip_{\loc}(M)$ is a \textit{weak IMCF in $M$ with initial condition $E_0$}, if $E_0=\{u<0\}$ and $u$ is a weak solution of the IMCF in $M\backslash\overline{E_0}$;
    \item such a $u$ is called \textit{proper}, if $\{u<t\}\Subset M$ for all $t\geq0$.
\end{itemize}
Combining Theorem \ref{thm-intro:main_ric_rd} or \ref{thm-intro:main_euc_isop} and the standard maximum principle, we have:

\begin{corollary}\label{cor-intro:ivp}
    Suppose the assumptions   of Theorem \ref{thm-intro:main_ric_rd} or \ref{thm-intro:main_euc_isop} hold. Let $E_0\Subset M$ be a $C^{1,1}$ domain. Then there exists a unique proper IMCF $u$ in $M$ with initial condition $E_0$. Moreover:
    \begin{enumerate}[label={(\roman*)}, topsep=1pt, itemsep=-0.5ex]
        \item Under the hypotheses of Theorem \ref{thm-intro:main_ric_rd}, we have
        \begin{equation}\label{eq-intro:lb_doubling}
            u(x)\geq(b-1)\log d(x,x_0)-C,\qquad\forall\,x\in M\backslash E_0.
        \end{equation}
        \item Under the hypotheses of Theorem \ref{thm-intro:main_euc_isop}, we have
        \begin{equation}\label{eq-intro:lb_euc_isop}
            u(x)\geq(n-1)\log d(x,x_0)-C,\qquad\forall\,x\in M\backslash E_0,
        \end{equation}
    \end{enumerate}
    Here, $x_0$ is any point in $E_0$ and $C<\infty$ is a constant independent of $x$.
\end{corollary}

\begin{proof}
    Let $v$ be the proper IMCF core with pole $x_0$, as given by Theorem \ref{thm-intro:main_ric_rd} or \ref{thm-intro:main_euc_isop}. By Huisken-Ilmanen's existence theorem \cite[Theorem 3.1]{huiskenilmanen} and the maximum principle \cite[Theorem 2.2]{huiskenilmanen}, there is a weak IMCF $u$ on $M$ with initial value $E_0$, such that $u\geq v-\sup_{E_0}(v)$ on $M\backslash E_0$. The corollary then follows from \eqref{eq-intro:growth_ric_rd} or \eqref{eq-intro:growth_euc_isop}.
\end{proof}

We now place the two results above into context.  An efficient way to produce IMCF is by approximation via renormalized $p$-harmonic functions, as first discovered by R. Moser \cite{moser}. For $1<p<\infty$, a function $v\in W^{1,p}_{\loc}$ is called \textit{$p$-harmonic} if 
\[
\Delta_pv:=\operatorname{div}\big(|\D v|^{p-2}\D v\big)=0
\]
weakly. In this case, $w := (1-p)\log v$ is then called a \textit{$p$-IMCF}. Note that $w$ solves
\[
\Delta_pw=|\D w|^p \qquad \text{weakly.}
\]
One of the main insights of R. Moser's discovery in \cite{moser}  is the following: if $\{v_p\}_{p>1}$ is a family of positive $p$-harmonic functions in a domain $\Omega$, such that
\[\lim_{p\to1}(1-p)\log v_p=u\quad\text{in}\quad C^0_{\loc}(\Omega)\]
for some $u\in\Lip_{\loc}(\Omega)$, then $u$ is a weak IMCF in $\Omega$. This idea was exploited by Moser himself \cite{moser} and later by B. Kotschwar and L. Ni \cite{kotschwarni} and L. Mari, M. Rigoli, and A. Setti \cite{ajm} in order to solve the initial value problem of the weak IMCF (see also \cite{cabe,gavi} for the anisotropic setting). The general strategy is as follows. Assume that $M$ is $p$-nonparabolic for each $p>1$ close enough to 1. Given $E_0$, we consider the minimal positive solution of the equation
\[\left\{\begin{aligned}
    & \Delta_pv_p=0\qquad\text{in}\ \ M\backslash E_0, \\
    & v_p=1\qquad\text{on}\ \ \p E_0.
\end{aligned}\right.\]
Since $M$ is $p$-nonparabolic, it follows that $0<v_p<1$ on $M\backslash\bar{E_0}$. Then we perform the transformation
\begin{equation}\label{eq-intro:moser_trans}
    u_p:=(1-p)\log v_p
\end{equation}
(so that $u_p=0$ on $\p E_0$, and $u_p>0$ on $M\backslash\bar{E_0}$), and take the limit
\begin{equation}\label{eq-intro:limiting}
    u=\lim_{p\to1}u_p,
\end{equation}
and finally argue that $u$ is a proper weak IMCF with initial condition $E_0$. Three main steps are involved in this program:
\begin{enumerate}[label={(\arabic*)}, topsep=1pt, itemsep=-0.5ex]
    \item One needs a gradient estimate that is stable as $p\to1$, so that the limit in \eqref{eq-intro:limiting} exists. For this step, see \cite[Theorems 1.1 and 3.1]{kotschwarni}, \cite[Lemma 2.17]{ajm} and \cite{sungwang}.
    \item Then by Moser's observation, the limit $u$ in \eqref{eq-intro:limiting} is a weak IMCF with initial value $E_0$.
    \item One needs a decay estimate for $v_p$ that is stable as $p\to1$, so that $u$ is proper.
\end{enumerate}
The last step is particularly subtle. This requires certain geometric conditions on $M$ at infinity, which often take the form of uniform Sobolev inequalities or volume growth and curvature bounds. 

Since we are also interested in approximating IMCF cores, we consider the following objects:
\begin{itemize}[topsep=1pt, itemsep=-0.5ex]
    \item The \textit{$p$-Green's kernel with pole $o$}, denoted by $\gr_p$, is the unique minimal positive solution of
    \[\Delta_p\gr_p=-\delta_o\]
    on $M$. For $p<n$, we have that $\gr_p$ is asymptotic to
    \[\omega_{n-1}^{\frac1{1-p}}\Big(\frac{p-1}{n-p}\Big)\,r(x)^{\frac{n-p}{1-p}}\]
    near $o$, where $\omega_{n-1}$ is the volume of the unit sphere $\mathbb{S}^{n-1}$; see Appendix \ref{sec:green}. Note that the existence of $\gr_p$ requires $M$ to be $p$-nonparabolic. When there is no risk of ambiguity, we shall often write $\gr$. We refer to Section \ref{sec:prelim} as well as to \cite[Section 2]{ajm} for more background on $p$-Green's kernels.
    \item The \textit{$p$-IMCF core with pole $o$}, denoted by $w_p$, is defined as
    \[w_p:=(1-p)\log\gr_p.\]
    Notice that $w_p$ is a $p$-IMCF in $M\backslash\{o\}$ as in the above definition.
\end{itemize}
% We will always assume $1<p<2$ unless otherwise specified.
%
In the previous work \cite{ajm}, the second to the fourth authors obtained two decay estimates for the Green's kernel $\gr_p$ assuming respectively the conditions of Theorem \ref{thm-intro:main_euc_isop} and \ref{thm-intro:main_ric_rd} (cf. respectively (3.44) and Theorem 3.23 in \cite{ajm}). The corresponding growth estimates for the IMCF core were obtained by taking these estimates as $p\to1$, as stated in \cite[Theorems 1.3 and 1.4]{ajm}. Then, following the proof of Corollary \ref{cor-intro:ivp} above, the usual initial value problem for the IMCF is then resolved (cf. \cite[Theorem 1.7]{ajm}).

These decay estimates for $\gr_p$ essentially boil down to \cite[Theorem 3.6]{ajm}, which states that a weighted Sobolev inequality of the form
\begin{equation}\label{eq-intro:pnu_solobev}
    \Big(\int_M\eta(r)^{-\frac p{\nu-p}}|\varphi|^{\frac{\nu p}{\nu-p}}\Big)^{\frac{\nu-p}\nu}\leq \So_{p,\nu}\int|\D\varphi|^p\qquad\forall\varphi\in\Lip_c(M)
\end{equation}
implies the corresponding growth estimate for the $p$-IMCF core
\begin{equation}\label{eq-intro:wp_growth}
    w_p(x)\geq (\nu-p)\log r(x)-C\big(\So_{p,\nu},p,\nu\big)-\log\eta(2r(x)),\quad\forall\,x\in M\backslash\{o\},
\end{equation}
for all $1<p<\nu$ and all positive non-decreasing weight functions $\eta$.

However, a gap was recently identified in the proof of \eqref{eq-intro:wp_growth}. As a result, one loses the uniform control of the constant $C$ in \eqref{eq-intro:wp_growth} as $p\to1$, and so the properness claim in Step (3) above breaks down. See Subsection \ref{subsec:the_gap} for a detailed account of this issue. One of the outcomes of our main theorems is to recover the properness results in \cite{ajm}. In particular, Theorem \ref{thm-intro:main_ric_rd} fully recovers \cite[Theorem 1.4 and 4.6]{ajm}, Corollary \ref{cor-intro:ivp} fully recovers \cite[Theorem 1.7]{ajm}, and Theorem \ref{thm-intro:main_euc_isop} supersedes \cite[Theorem 1.3]{ajm}. See Subsection \ref{subsec:ajm} for a detailed account. We underline that no curvature lower bound is required in Theorem \ref{thm-intro:main_euc_isop} (although at the cost of losing uniform gradient estimates for $u$). The validity of the decay estimate stated in \cite[Theorem 3.6]{ajm} itself remains open.

%Regarding \cite[Theorem 1.3]{ajm}, 
We note that Theorem \ref{thm-intro:main_euc_isop} proves the optimal growth estimate for the IMCF core, as can be seen by considering the case of Euclidean space. Previously, the existence of a unique proper flow under the assumptions in Theorem \ref{thm-intro:main_euc_isop} was obtained in \cite{Xu_2023_proper}, but the growth estimate was non-optimal.

On the other hand, given the setups in Theorem \ref{thm-intro:main_euc_isop}, it remains open \textit{whether the IMCF core $u$ is the limit of $p$-IMCF cores}. Namely, the following is currently unknown:

\begin{question}\label{qs-intro:p_convergence}
    Suppose $M$ satisfies \eqref{eq-intro:euc_isop}, and for some sequence $p_i\to1$ and $w\in\Lip_{\loc}(M\backslash\{o\})$ we have the convergence
    \[w_{p_i}\to w\quad\text{in}\quad C^0_{\loc}(M\backslash\{o\}),\]
    where $w_{p_i}$ is the $p_i$-IMCF core with pole $o$. Is $w$ equal to the proper IMCF core with pole $o$ (up to additive constants)?
\end{question}

A positive answer of Question \ref{qs-intro:p_convergence} would allow us to fully recover \cite[Theorem 1.3]{ajm}. It is also natural to conjecture the following uniform growth estimate, which would answer Question \ref{qs-intro:p_convergence} affirmatively:

\begin{question}\label{qs-intro:p_growth}
    Suppose $M$ satisfies \eqref{eq-intro:euc_isop}. Is it true that for all $p\in(1,2)$, we have
    \[w_p(x)\geq(n-p)\log r(x)-C(n,c_I)\ \text{?}\]
\end{question}

Note that Question \ref{qs-intro:p_growth} is a special case of \cite[Theorem 3.6]{ajm}.

\subsection{General \texorpdfstring{$p$-harmonic}{p-harmonic} approximation of IMCF}

We include here some discussions on the global and interior approximation of the weak IMCF, which may be of independent interest.

In Theorem \ref{thm-intro:main_euc_isop}\,--\,\ref{cor-intro:ivp} above, we see that weak IMCFs can be approximated by $p$-IMCFs, assuming certain conditions on $M$. This need not be the case, however, if no additional assumption is made (for instance, $M$ could admit a proper IMCF but be $p$-parabolic for all $p>1$). A way to partially get around this issue is to try approximating a weak IMCF in bounded domains. A recent idea of L. Benatti, A. Pluda and M. Pozzetta \cite{BPP24} led to the following approximation theorem:

\begin{theorem}[{\cite[Theorem 2.8]{BPP24}}]\label{thm-intro:BPP}
    Suppose that $M$ is complete, noncompact, and that there exists a proper IMCF $u$ on $M$ with initial value $E_0\Subset M$. Then for all domain $D$ with $E_0\Subset D\Subset M$, there is a family of $p$-IMCFs $u_p$ on $\bar D\backslash E_0$, such that
    \begin{itemize}[topsep=1pt, itemsep=-0.6ex]
        \item $u_p\geq0$ and $u_p|_{\p E_0}=0$,
        \item $\lim_{p\to1}u_p=u$ in $C^0(\overline{D\backslash E_0})$.
    \end{itemize}
\end{theorem}

To summarize: for any $D$, one can find a ($D$-dependent) sequence of $p$-IMCFs that converges to $u$ (at a rate that may not be uniform when $D$ varies). This result allows one to transfer analytic properties of $p$-harmonic functions to the limit, see \cite{BPP24} for specific applications. It is also powerful enough to lead to Theorem \ref{thm-intro:obstacle} below.

However, it is crucial that $u$ be proper here, so this is essentially a global approximation theorem. Its interior counterpart remains largely open: for instance, one can ask

\begin{question}\label{qs-intro:loc_approx_ivp}
    Does the result of Theorem \ref{thm-intro:BPP} hold if $u$ is not proper?
\end{question}

One can even drop the initial condition and ask the following:

\begin{question}\label{qs-intro:loc_approx}
    Suppose $u\in\Lip_{\loc}(\Omega)$ is a weak IMCF in a domain $\Omega$. Given a sub-domain $D\Subset\Omega$, can we find a family $\{u_p\}_{p>1}$ of $p$-IMCFs, so that
    \[u_p\to u\qquad\text{in}\ \ C^0(\overline D)\,?\]
\end{question}

As an application, this would imply the following gradient estimate: if $u$ is a weak IMCF in a ball $B(x,R)\Subset M$, and we have $\sec\geq-K$ in $B(x,R)$, then it holds $|\D u(x)|\leq C(n,KR^2)R^{-1}$ (by Kotschwar-Ni \cite{kotschwarni}). We remark that although weak IMCFs are defined to be locally Lipschitz, no a priori gradient estimates have been obtained in general (we only have gradient estimates for those solutions that are limits of $p$-IMCFs or limits of solutions of the elliptic regularized equations).

It is natural to attack Question \ref{qs-intro:loc_approx} by considering the following Dirichlet problem
\[\left\{\begin{aligned}
    & \Delta_pu_p=|\D u_p|^p\qquad\text{in}\ \ \Omega, \\
    & u_p=u\qquad\text{on}\ \ \p\Omega,
\end{aligned}\right.\]
and hope that $u_p\to u$ as $p\to1$. However, $u_p$ could converge to an IMCF that is different from $u$. To our knowledge, there is currently no result guaranteeing that the Dirichlet condition is preserved as $p\to1$.

The limiting behavior of Dirichlet conditions is an interesting question in its own right. Consider the case of $p$-capacitors: suppose $\Omega\Subset M$ is a smooth domain, and $E_0\Subset\Omega$ is a $C^{1,1}$ domain. Let $v_p$ be the unique solution of the equation
\begin{equation}\label{eq-intro:capacitary}
    \left\{\begin{aligned}
        & \Delta_p v_p=0\qquad\text{in }\Omega\backslash\overline{E_0}, \\
	& v_p=1\qquad\text{on }\p E_0, \\
        & v_p=0\qquad\text{on }\p\Omega,
    \end{aligned}\right.
\end{equation}
make  Moser's transformation $u_p=(1-p)\log v_p$, and take the limit
\begin{equation}\label{eq-intro:potential_convergence}
    u=\lim_{p\to1}u_p\qquad\text{in}\ \ C^0_{\loc}(\Omega\backslash E_0).
\end{equation}
It follows from standard theory that $u$ is a weak IMCF in $\Omega$ with initial value $E_0$. the remaining question is what boundary condition is satisfied on $\p\Omega$. Our next result states that $u$ satisfies the outer obstacle condition which was developed by the fifth author in \cite{Xu_2024_obstacle}.

\begin{theorem}\label{thm-intro:obstacle}
    In the above setup, $u$ coincides with the solution of IMCF in $\Omega$ with initial value $E_0$ and outer obstacle $\p\Omega$, obtained in \cite[Theorem 1.6]{Xu_2024_obstacle}.
\end{theorem}

The outer obstacle condition, morally (but less precisely) speaking, states that each level set of $u$ contacts tangentially with $\p\Omega$. The convergence \eqref{eq-intro:potential_convergence} is subtle: each $u_p$ ($p>1$) diverges to $+\infty$ at $\p\Omega$, but the limit solution $u$ is bounded on $\Omega$. All but boundedly many level sets of $u_p$ disappear in the limiting process (due to Theorem \ref{thm-intro:obstacle}), and the boundary tangency is preserved through the limit.

Our current proof of Theorem \ref{thm-intro:obstacle} combines Theorem \ref{thm-intro:BPP} and the soft obstacle approximation argument in \cite{Xu_2024_obstacle}. More in detail, one sees that this proof is mostly IMCF-theoretic. Specifically, we comment that:
\begin{itemize}[nosep]
    \item The main nonlinear potential-theoretic fact that we used is the following: on manifolds with conic ends, proper IMCFs are globally approximated by $p$-IMCFs (this underlies the original proof of Theorem \ref{thm-intro:BPP});
    \item The Dirichlet condition $v_p|_{\p\Omega}=0$ is used only in the following place: $v_p$ is the smallest $p$-harmonic function in $\Omega\backslash E_0$ with $v_p|_{\p E_0}=1$.
\end{itemize}
It would be interesting to find a potential-theoretic proof of Theorem \ref{thm-intro:obstacle}. For this goal, perhaps one has to study the fine properties of the vector field $|\D u_p|^{p-2}\D u_p$ near $\p\Omega$. Such an argument, if found, may ideally imply that $|\D u_p|^{p-2}\D u_p$ converges to a boundary-orthogonal calibration of $u$ (see \cite[Section 3.4]{Xu_2024_obstacle} for the latter notion).

\vspace{12pt}

\textbf{Acknowledgements.} 
We would like to thank Gerhard Huisken for conversations on the IMCF core. L.M. and M.R. are supported by the PRIN project no. 20225J97H5 ``Differential-geometric aspects of manifolds via Global Analysis''. A.G.S. and L.B. are members of the GNAMPA INdAM group. This research was funded in part by the Austrian Science Fund (FWF) [grant DOI \href{https://www.fwf.ac.at/en/research-radar/10.55776/EFP6}{10.55776/EFP6}]. For open access purposes, the authors have applied a CC BY public copyright license to any author-accepted manuscript version arising from this submission.

\section{Preliminaries}\label{sec:prelim}

For the preliminary material in this section, we refer the reader to \cite{ajm} and references therein. For $p \in (1,n]$, the Green kernel $\gr$ of $\Delta_p$ on a smooth, relatively compact domain $\Omega \Subset M$ with pole at $o \in \Omega$ and Dirichlet boundary conditions on $\partial \Omega$ is the minimal positive solution to 
\begin{equation}\label{eq_distri}
\Delta_p \gr = -\delta_o. 
\end{equation}
The construction of $\gr$ can be found in \cite[Thm. 3.19]{holopainen}. There, the author also proves the following properties:
\begin{equation}\label{eq_levelGreen}
\int_{ \{\gr \in [s,t]\} } |\nabla \gr|^p = t-s \qquad \forall \, 0 < s < t
\end{equation}
and 
\begin{equation}\label{eq_nice_pcapac}
\capac_p\big( \{\gr \ge \ell\}, \Omega\big) = \ell^{-p} \int_{\{\gr \le \ell\}} |\nabla \gr|^p = \ell^{1-p}, 
\end{equation}
where the $p$-capacity $\capac_p(K,\Omega)$ for $K$ compact and $\Omega$ open is defined by 
$$
\capac_p(K,\Omega) = \inf \left\{ \int_\Omega |\nabla \psi|^p \ : \ \psi \in \lip_c(\Omega), \ \psi\ge 1 \ \text{\,on\,} \, K\right\}.
$$
The minimality of the kernel constructed in \cite{holopainen} was later proved in \cite{ajm} as a consequence of a comparison theorem for Green kernels.

If $\Omega$ has non-compact closure or nonsmooth boundary (including $\Omega=M$), a hernel $\gr$ with pole at $o$ was constructed in \cite{holopainen} by means of an increasing exhaustion $\{\Omega_j\}$ of smooth domains of $\Omega$ and related Green kernels $\gr_{j}$. The existence of a locally finite limit for the family $\{\gr_j\}$ is not automatic, and it is equivalent to the fact that $\Delta_p$ is non-parabolic on $\Omega$, namely, that $\capac_p(K,\Omega) > 0$ for some (equivalently all) $K$ with non-empty interior.

The behavior of $\gr$ in a neighbourhood of $o$ was described by J. Serrin in \cite{Serrin_2}, and later refined by Kichenassamy and Veron in \cite{kichenveron} and \cite[pp. 243-251]{veron}. Their proof carries over to manifolds, as stated in \cite[Theorem 2.4]{ajm}, leading to the following theorem. Setting 
\begin{equation}\label{def_mup}
    \disp \mu(r) = \left\{ \begin{array}{ll}
        \disp \omega_{n-1}^{- \frac{1}{p-1}}\left(\frac{p-1}{n-p}\right) r^{- \frac{n-p}{p-1}} & \qquad \text{if } \, p < n \\[0.5cm]
        \disp \omega_{n-1}^{- \frac{1}{n-1}}(-\log r) & \qquad \text{if } \, p=n,
    \end{array}\right.
\end{equation}
it holds
\begin{theorem}\label{teo_localsingular}
For $p \le n$, let $\gr$ be a Green kernel for $\Delta_p$ on an open set $\Omega \subset M^n$ containing $o$. Then, $\gr$ is smooth in a punctured neighbourhood of $o$ and, as $x \ra o$,
\begin{equation}
    \begin{array}{ll}
        (1) & \quad \gr \sim \mu(r), \\[0.2cm]
        (2) & \quad | \nabla \gr - \mu'(r) \nabla r | = o\big( \mu'(r)\big),\\[0.3cm]
        (3) & \quad \text{if } \, p < n, \quad \big| \nabla^2 \gr - \mu''(r) \di r \otimes \di r - \frac{\mu'(r)}{r} \big( \metric - \di r \otimes \di r\big) \big| = o\big( \mu''(r)\big).
    \end{array}
\end{equation}
\end{theorem}
Since, to our knowledge, its proof only appeared in an unpublished version of \cite{ajm} on arXiv, we include it in the Appendix.

\subsection{Doubling, \texorpdfstring{$(q,p)$}{(q,p)}-Poincar\'e and Sobolev inequalities}\label{sec_proper}

\begin{definition}\label{def-prelim:vd_poincare}
We say that an open subset $\Omega \subset M^n$ satisfies:
\begin{itemize}
    \item[$\dous$] the \emph{doubling property} if there exists a constant $C_{\Dou} > 1$ such that
    $$
    |B_{2r}(x)| \le C_{\Dou} |B_r(x)| \qquad \text{for each } \ \  B_{2r}(x) \Subset \Omega.
    $$
    \item[$\neuqp$] the \emph{weak $(q,p)$-Poincar\'e inequality}, for given $1 \le q \le p < \infty$, if there exists a constant $\Po_{q,p}$ such that
    \begin{equation}\label{WNP}
    \left\{ \fint_{B_r(x)} |\psi - \bar \psi_{B_r(x)}|^q\right\}^{\frac{1}{q}} \le \Po_{q,p} r \left\{ \fint_{B_{2r}(x)} |\nabla \psi|^p \right\}^{\frac{1}{p}}
    \end{equation}
    holds for each $B_{2r}(x) \Subset \Omega$ and $\psi \in \lip_c(\Omega)$.
\end{itemize}
\end{definition}
\begin{remark}
	As shown in \cite[Lem. 8.1.13]{hkst}, $\dous$ implies:
\begin{equation}\label{rel_low_vol}
	\forall \, B_s' \subset B_t \subset B_{2t} \Subset \Omega \ \text{balls}, \quad \frac{|B_s'|}{|B_t|} \ge C \left(\frac{s}{t}\right)^\nu
\end{equation}
with 
\[
\nu =  \log_2 C_\Dou, \qquad C = C(C_\Dou).
\]
The constant $\nu$ is called the doubling dimension of $M$. 
\end{remark}

\begin{remark}
	Note that, by the H\"older inequality,  the validity of $\neuqp$ implies that of $(\mathrm{NP}_{q',p'})$ for each $q' \le q$ and $p' \ge p$. 
\end{remark}

\begin{example}\label{ex_mrs_local}\cite[Remark 3.1]{ajm}
Let $B_{6R} \Subset M^n$ be a relatively compact geodesic ball centered at a fixed origin, and suppose that
\begin{equation}\label{eq_iporicci_local}
\Ricc \ge -(n-1)\kappa^2 \qquad \text{on } \, B_{6R},
\end{equation}
for some constant $\kappa \ge 0$. Denote by $v_\kappa(t)$ (respectively, $V_\kappa(t)$) the volume of a sphere (resp. ball) of radius $t$ in the space form of curvature $-\kappa^2$. Then, $B_R$ satisfies:
\begin{itemize}
    \item the doubling property with constant $C_{\Dou} \doteq \frac{V_\kappa(2R)}{V_\kappa(R)}$.
    In particular, if $\Ricc \ge 0$ then the entire $M$ satisfies $\dous$; 
    \item $\neup$ with constant 
    \begin{equation}\label{poincare_constant}
    \Po_{p,p} = \exp\left\{ \frac{c_n(1+ \kappa R)}{p}\right\}
    \end{equation}
    and $c_n$ only depending on $n$;
    \item the $L^1$ Sobolev inequality
    \begin{equation}\label{eq_sobosempote}
    \begin{array}{lcl}
    \disp \left(\int |\psi|^{\frac{n}{n-1}} \right)^{\frac{n-1}{n}} \le \So_{1,n}(R) \int |\nabla \psi| \qquad \forall \, \psi \in \lip_c(B_R)
    \end{array}
    \end{equation}
    with constant 
    \[
    \So_{1,n}(R) \le \disp \frac{C_{n}V_\kappa(2R)}{|B_R|} \left( \Po_{1,1} + \frac{V_\kappa(3R)}{R v_\kappa(R)} \right),   
    \]
    and the $L^p$ Sobolev inequality
    \begin{equation}\label{SobolevL1_local}
    \left( \int |\psi|^{\frac{np}{n-p}} \right)^{\frac{n-p}{n}} \le \So_{p,n}(R) \int |\nabla \psi|^p \qquad \forall \, \psi \in \lip_c(B_{R})
    \end{equation}
    with constant
    $$
    \So_{p,n}(R) = \left[ \frac{\So_{1,n}(R) p(n-1)}{n-p}\right]^p \to \So_{1,n}(R) \qquad \text{as } \, p \to 1.
    $$
\end{itemize}
\end{example}

We shall be interested in manifolds globally satisfying $\dous$, $\neuuno$. Note that this is the case if $\Ric \ge 0$ on $M$. By \cite[Thm. 9.1.15]{hkst} and in view of \eqref{rel_low_vol}, the assumption guarantees the existence of a constant $\So(C_\Dou,\Po_{1,1})$ such that a Neumann $L^1$-Sobolev inequality holds: 
 
\begin{equation}\label{neusob}
	\left( \fint_{B_r(y)} |\psi - \bar\psi_{B_r(y)}|^{\frac{\nu}{\nu-1}} \right)^{\frac{\nu-1}{\nu}} \le \So r \fint_{B_r(y)} |\nabla \psi| \qquad \forall \, y \in M, \psi \in \lip(B_r(y)),
\end{equation}

where $\nu$ is the doubling dimension in \eqref{rel_low_vol}. It is shown in \cite[Lemma 3.21]{ajm}, that \eqref{neusob} implies the Neumann $L^p$-Sobolev 
	\begin{equation}\label{neusob_p}
		\left( \fint_{B_r(y)} |\psi - \bar\psi_{B_r(y)}|^{\frac{\nu p}{\nu-p}} \right)^{\frac{\nu-p}{\nu}} \le \bar \So r \left(\fint_{B_r(y)} |\nabla \psi|^p\right)^{\frac{1}{p}} \quad \forall \, y \in M, \psi \in \lip(B_r(y)),
	\end{equation}
for each $1 \le p \le p_0 < \nu$, with constant $\bar \So$ only depending on $C_\Dou, \Po_{1,1}, p_0$. 
\begin{remark}\label{rem_neup}
In particular, \eqref{neusob_p} together with the H\"older inequality guarantee that a manifold satisfying $\dous$ and $\neuuno$ also supports $\neup$ for each $p \in [1,p_0]$, with a constant $\Po_{p,p}$ only depending on $C_\Dou, \Po_{1,1}, p_0$.
\end{remark}

Inequality \eqref{neusob} and $\dous$ yield local Sobolev inequalities with uniform constants. Since we found no precise reference, we briefly supply the argument, that was suggested to us by P. Koskela. 

\begin{proposition}\label{prop_local_sobolev_dir}
    Let $M$ be a complete noncompact manifold satisfying $\dous$ and $\neuuno$ for some constants $C_\Dou$ and $\Po_{1,1}$. Then, there exists a constant $\hat \So(C_\Dou, \Po_{1,1})$ such that, for each ball $B_r = B_r(x)$,
    
    \begin{equation}\label{isop_mean}
    	\left( \fint_{B_r} |\psi|^{\frac{\nu}{\nu-1}} \right)^{\frac{\nu-1}{\nu}} \le \hat \So r \fint_{B_r} |\nabla \psi| \qquad \forall \, \psi \in \lip_c(B_r),
    \end{equation}
\end{proposition}

\begin{proof}
Fix a point $y \in \partial B_{2r}$, so that $B_r(y) \subset B_{3r}\backslash B_r$. Using \eqref{neusob} on $B_{4r}(y)$ and the triangle inequality we get
\begin{equation}\label{ste1}
	\begin{array}{lcl}
		\disp \left( \fint_{B_{4r}(y)} |\psi|^{\frac{\nu}{\nu-1}} \right)^{\frac{\nu-1}{\nu}} & \le & \disp \left( \fint_{B_{4r}(y)} |\bar \psi_{B_{4r}(y)}|^{\frac{\nu}{\nu-1}} \right)^{\frac{\nu-1}{\nu}} + 4\So r \fint_{B_{4r}(y)} |\nabla \psi| \\[0.5cm]
		& = & \disp \fint_{B_{4r}(y)} |\psi| + 4\So r \fint_{B_{4r}(y)} |\nabla \psi|.
	\end{array}
\end{equation}		
Denote by $A_{r}(y) = B_{4r}(y)\backslash B_r(y)$. Since $\psi = 0$ on $B_r(y)$, using $\dous$ and the H\"older inequality gives
\[
\begin{array}{lcl}
	\disp \fint_{B_{4r}(y)} |\psi| & = & \disp \frac{|A_{r}(y)|}{|B_{4r}(y)|} \fint_{A_r(y)} |\psi| \le \frac{|A_{r}(y)|}{|B_{4r}(y)|} \left( \fint_{A_r(y)} |\psi|^{\frac{\nu}{\nu-1}} \right)^{\frac{\nu-1}{\nu}} \\[0.5cm]
	& = & \disp \left[\frac{|A_{r}(y)|}{|B_{4r}(y)|}\right]^{\frac{1}{\nu}} \left( \fint_{B_{4r}(y)} |\psi|^{\frac{\nu}{\nu-1}} \right)^{\frac{\nu-1}{\nu}} \le \disp \left[\frac{C_\Dou^2 -1}{C_\Dou^2}\right]^{\frac{1}{\nu}} \left( \fint_{B_{4r}(y)} |\psi|^{\frac{\nu}{\nu-1}} \right)^{\frac{\nu-1}{\nu}}.
\end{array}
\]
Inserting into \eqref{ste1}, we get
\[
\left( \fint_{B_{4r}(y)} |\psi|^{\frac{\nu}{\nu-1}} \right)^{\frac{\nu-1}{\nu}}	\le \left[ 1 - \left(\frac{C_\Dou^2 -1}{C_\Dou^2}\right)^{\frac{1}{\nu}}\right]^{-1} 4\So r \fint_{B_{4r}(y)} |\nabla \psi|.
\]
Call $\So_1$ the constant in the right-hand side. Concluding, since $\psi$ is supported on $B_r$ and by \eqref{rel_low_vol} we obtain
\[
\left( \fint_{B_{r}} |\psi|^{\frac{\nu}{\nu-1}} \right)^{\frac{\nu-1}{\nu}} \le \left[\frac{B_{4r}(y)}{B_r}\right]^{\frac{\nu-1}{\nu}} \left( \fint_{B_{4r}(y)} |\psi|^{\frac{\nu}{\nu-1}} \right)^{\frac{\nu-1}{\nu}}	\le \So_1 r \left[ \frac{|B_r|}{|B_{4r}(y)|}\right]^{\frac{1}{\nu}}\fint_{B_r} |\nabla \psi|.
\]
Inequality \eqref{isop_mean} readily follows since $B_r \subset B_{4r}(y)$.
\end{proof}

Manifolds supporting $\dous$, $\neup$ and a reverse doubling inequality enjoy a global, weighted Sobolev inequality. The next theorem is due to V. Minerbe  \cite{minerbe} for $p=2$ and to D. Tewodrose \cite{tewo2} for general $p$. We quote the following simplified version of \cite[Thm 1.1]{tewo2}.

\begin{theorem}
	\label{theorem-tewo2}
	Let $M^n$ be a complete manifold satisfying $\dous$ and $\neup$ for some $p\in[1,\nu)$, with doubling dimension $\nu=\log_2 C_\Dou$. Assume that there exist constants $C_{\RD}>0$, $b\in (p, \nu]$ and a point $o \in M$ such that
	\begin{equation}\label{eq_reversevol}
		\forall \, t \ge s > 0, \qquad \frac{|B_t(o)|}{|B_s(o)|} \ge C_{\RD} \left( \frac{t}{s}\right)^b.
	\end{equation}
	Then, there exists $\So_{p,\nu}$ depending only on $C_\Dou$, $p$, $\Po_{p,p}$, $b$ and $C_{\RD}$ such that
	\begin{equation}
		\label{sob_tewodrose}
		\left(
		\int_M \left[ \frac{r^\nu}{|B_r(o)|}
		\right]^{-\frac{p}{\nu-p} } |\psi|^{\frac{\nu p}{\nu-p}}
		\right)^{\frac{\nu-p}{\nu}}
		\leq
		\So_{p,\nu} \int_M |\nabla \psi|^p, \qquad \forall \,\psi\in \lip_c (M).
	\end{equation}
\end{theorem}
\begin{remark}\label{rem_Pp}
	If we assume $\neuuno$ in Theorem \ref{theorem-tewo2}, then by Remark \ref{rem_neup} the constant $\Po_{p,p}$ remains uniformly bounded on intervals $[1,p_0]$ for fixed $p_0 < \nu$. Moreover, if \eqref{eq_reversevol} holds for some $b \in (1,\nu]$, by the discussion on pages 828-831 in \cite{ajm}, the constant $\So_{p,\nu}$ can be uniformly bounded in terms of $C_\Dou, \Po_{1,1}, \nu, b, C_\RD$ and $p_0$.
\end{remark}

\begin{remark}\label{rem_reversevol}
	Condition \eqref{eq_reversevol} holds, for instance, if the balls $B_t$ centered at a fixed origin $o$ satisfy
		$$
		\bar C^{-1} t^b \le |B_t| \le \bar C t^b \qquad \forall \, t \ge 1.
		$$
        for some constant $\bar C>1$ and with $b \in (p,m]$. In this case, the reverse doubling constant $C_\RD$ can be chosen to depend on $\bar C, m$ and on a lower bound $H$ for the Ricci curvature on, say, $B_{6}$. 
\end{remark}

\subsection{Mean value and Harnack inequalities}

The following lemmas are extracted from \cite{ajm}.

\begin{lemma}[{\cite[Lemma 3.3]{ajm}}]\label{lem_moser}
    Let $A_\infty \Subset M$ and fix $T>0$ in such a way that $A_0 = B_T(A_\infty) \Subset M$. Suppose that the following Sobolev inequality holds:
    \begin{equation}\label{sobolev_permoser}
    \left( \int |\psi|^{ \frac{\nu p}{\nu-p}}\right)^{\frac{\nu-p}{\nu}} \le \So_{p,\nu} \int |\nabla \psi|^p \qquad \forall \, \psi \in \lip_c(A_0),
    \end{equation}
    for some $p \in (1,\infty)$, $\nu>p$ and $\So_{p,\nu}>0$. Fix $q \ge p$. Then, for each weak solution $0 \le u \in C(A_0)\cap W^{1,p}_\loc(A_0)$ to $\Delta_p u \ge 0$ it holds
    \begin{equation}\label{halfhar_sub}
    	\sup_{A_\infty} u \le \disp  (\So_{p,\nu} \bar C_{p,\nu})^{\frac{\nu}{p q}} T^{-\frac{\nu}{q}} |A_0|^{\frac{1}{q}} \left(\fint_{A_0}  u^{q}\right)^{\frac{1}{q}},
    \end{equation}
    with
    \begin{equation}\label{Cpnu_mag0}
    	\bar C_{p,\nu} = 2^{\nu} \left[ 1+ p\right]^p.
    \end{equation}
\end{lemma}

\begin{theorem}[{\cite[Theorem 3.4]{ajm}}]\label{teo_harnack}
Fix $p \in (1, \infty)$. Let $u$ be a positive solution of $\Delta_p u=0$ on a ball $B_{6R} = B_{6R}(x_0)$, and suppose that the following weak $(1,p)$-Poincar\'e inequality holds on $B_{4R}$, for some constant $\Po_{1,p}$:
\begin{equation}\label{poinca_p}
\fint_{B_r(y)} |\psi - \bar \psi_{B_r(y)}| \le \Po_{1,p} r \left\{ \fint_{B_{2r}(y)} |\nabla \psi|^p\right\}^{\frac{1}{p}} \qquad \forall \, \psi \in \lip_c(B_{4R}),
\end{equation}
for every ball $B_r(y)\Subset B_{2R}$. Assume the validity of the Sobolev inequality
$$
\left( \int |\psi|^{ \frac{\nu p}{\nu-p}}\right)^{\frac{\nu-p}{\nu}} \le \So_{p,\nu} \int |\nabla \psi|^p \qquad \forall \, \psi \in \lip_c(B_{4R}),
$$
for some $\nu>p$ and $\So_{p,\nu}>0$. Then, having fixed $p_0 \in (p,\nu)$, the following Harnack inequality holds:
$$
\sup_{B_R} u \le \Ha_{p,\nu}^{\frac{1}{p-1}} \inf_{B_R} u,
$$
with constant
\begin{equation}\label{ine_harnack_3}
\Ha_{p,\nu} = \exp\left\{ c_2 \Po_{1,p} \left[\frac{|B_{6R}|}{|B_{2R}|}\right]^{\frac{1}{p}} Q^{-2} p \right\},
\end{equation}
where $c_2>0$ is a constant depending only on $\nu$ and $p_0$,
$$
Q = \inf_{\tau \in [1, \frac{\nu}{\nu-p}]} \big( \So_{p,\nu} C_{p,\nu} \big)^{- \frac{\nu \tau}{p}} R^{\nu \tau} |B_{2R}|^{-\tau}
$$
and
\begin{equation}\label{Cpnu}
C_{p,\nu} = 2^{\nu} \max \left\{ [1+p]^p, \frac{3^p \nu^\nu}{p^p(\nu-p)^{\nu-p}}\right\}.
\end{equation}
\end{theorem}

The next result is to show that, by assuming the validity of $\dous$, $\neuuno$ on the entire manifold $M$, we are able to uniformly bound the Harnack constant and get a sharp control on its growth as $p \to 1$.
\begin{theorem}\label{teo_uniHarnack}
    Let $M^n$ be a complete manifold satisfying $\dous,\neuuno$ with doubling dimension $\nu = \log_2 C_\Dou > 1$. Fix $p_0 \in (1, \nu)$. Then, there exist constants $\mathscr{C},\mathscr{H}$ depending on $C_\Dou,p_0$ and $\Po_{1,1}$ such that, for each $p \in (1, p_0]$ and each ball $B_{6R} \Subset M$,
    \begin{itemize}
        \item[-] every solution $0 \le u \in C(B_{2R})\cap W^{1,p}(B_{2R})$ to $\Delta_p u \ge 0$ satisfies
	\begin{equation}\label{eq_weakHar}
            \sup_{B_R} u \le \mathscr{C} \left( \fint_{B_{2R}} u^p \right)^{\frac{1}{p}}
	\end{equation}
        \item[-] every $p$-harmonic function $u>0$ on $B_{6R}$ satisfies $\sup_{B_R} u \le \mathscr{H}^{\frac{1}{p-1}} \inf_{B_R} u$.
    \end{itemize}
\end{theorem}

\begin{proof}
	We apply Proposition \ref{prop_local_sobolev_dir} and plug the test function $|\psi|^{\frac{p(\nu-1)}{\nu-p}}$ into \eqref{isop_mean} to deduce, after an application of the H\"older inequality, 
	\[
	\left( \int_{B_r} |\psi|^{\frac{\nu p}{\nu-p}} \right)^{\frac{\nu-p}{\nu}} \le \left[\frac{\hat \So p(\nu-1)}{\nu-p}\right]^p |B_r|^{-\frac{p}{\nu}}r^p \int_{B_r} |\nabla \psi|^p \qquad  \forall \, \psi \in \lip_c(B_r).
	\]
	Setting $r = 4R$ and
	\[	
	\So_{p,\nu} = \left[\frac{\hat \So p(\nu-1)}{\nu-p}\right]^p |B_{4R}|^{-\frac{p}{\nu}}(4R)^p,
	\]
	again by the doubling property we obtain that the constant $\Ha_{p,\nu}$ in \eqref{ine_harnack_3}, as well as the constant in the right hand side of \eqref{halfhar_sub} (with $q = p$, $A_\infty = B_R$, $T = R$) can be bounded from above by constants depending on $p_0$, $C_\Dou$ and $\Po_{1,1}$. The theorem follows from Lemma \ref{lem_moser} and Theorem \ref{teo_harnack}.
\end{proof}

\section{Decay estimates for the Green kernel}\label{sec:p_imcf}

In \cite{ajm}, the authors claimed a certain decay estimate, see Conjecture \ref{teo_sobolev} below, for the Green kernel $\gr$ of $\Delta_p$ on an open subset $\Omega \subset M$ supporting a weighted Sobolev inequality
\begin{equation}\label{sobolev_weighted}
\left( \int_{\Omega} \eta(r)^{-\frac{p}{\nu-p}} |\psi|^{ \frac{\nu p}{\nu-p}}\right)^{\frac{\nu-p}{\nu}} \le \So_{p,\nu} \int_{\Omega} |\nabla \psi|^p \qquad \forall \, \psi \in \lip_c(\Omega).
\end{equation}

Here, $\eta$ is a function satisfying
\begin{equation}\label{def_eta_sobolev}
    \disp \eta \in C(\R^+_0), \qquad \eta > 0, \qquad \eta(t) \ \ \text{ non-decreasing on } \, \R^+_0
\end{equation}
and $\So_{p,\nu}$ also depends implicitly on $\eta$. Let $o \in M$ and set $r(x) = \dist(x,o)$. The claimed estimate is stated as follows:

\begin{conjecture}\label{teo_sobolev}
Let $\Omega \subset M$ be a connected open set, and denote by $r$ the distance from a fixed origin $o \in \Omega$. Assume that $\Omega$ supports the weighted Sobolev inequality \eqref{sobolev_weighted} for some $p \in (1,\nu)$, constant $\So_{p,\nu}>0$ and weight $\eta$ satisfying \eqref{def_eta_sobolev}. Then, $\Delta_p$ is non-parabolic on $\Omega$ and, letting $\gr(x)$ be the Green kernel of $\Delta_p$ on $\Omega$ with pole at $o$,
\begin{equation}\label{upper_consobolev}
\gr(x) \le C_{p,\nu}^{\frac{1}{p-1}} \eta\big(2r(x)\big)^{\frac{1}{p-1}} r(x)^{- \frac{\nu-p}{p-1}}, \qquad \forall \, x \in  \Omega \backslash \{o\},
\end{equation}
where
$$
C_{p,\nu} = \So_{p,\nu}^{\frac{\nu}{p}} \left[ 2^{\nu}p(1+p)^p \left(\frac{p}{p-1}\right)^{p-1}\right]^{\frac{\nu-p}{p}}
$$
is bounded as $p \ra 1$ if so is $\So_{p,\nu}$. In particular, if
\begin{equation}\label{ipo_eta_2}
\eta(t) = o \left( t^{\nu - p} \right) \qquad \text{as } \, t \ra \infty.
\end{equation}
then
 $\gr(x) \ra 0$ as $r(x) \ra \infty$ in $\Omega$.
\end{conjecture}

In Subsection \ref{subsec:the_gap} we describe the gap in \cite{ajm} in the proof of Conjecture \ref{teo_sobolev}. In Subsection \ref{subsec:new_est}, we prove a new decay estimate of the $p$-Green's kernel, see Theorem \ref{theorem-decayGreen2}, which will lead to the proof of Theorem \ref{thm-intro:main_ric_rd}.

\subsection{Description of the gap in \texorpdfstring{\cite{ajm}}{[10]}}\label{subsec:the_gap}

Conjecture \ref{teo_sobolev} is stated as Theorem 3.6 in \cite{ajm}. The argument there is based on a Moser iteration procedure for the function $\|\gr\|_t = \max_{\Omega \cap B_t}\gr$, which is shown in \cite[(3.38)]{ajm} to satisfy  

\begin{equation}\label{sobolev_iteriamola_00}
\|\gr\|_t \le \hat C (\theta t)^{- \frac{\nu-p}{p}} \|\gr\|_{(1-\theta)t}^{\frac{1}{p}},\qquad\forall\,\theta\in(0,1),
\end{equation}
where
\begin{equation}\label{def_Chat}
    \hat C = \So_{p,\nu}^{\frac{\nu}{p^2}} \bar C_{p, \nu}^{\frac{\nu-p}{p^2}} \eta\big(2t\big)^{\frac{1}{p}} 
\end{equation}
and $\bar C_{p,\nu}$ is defined in \eqref{Cpnu_mag0}. For a chosen sequence $t=t_0>t_1>t_2>\cdots$, setting $\theta=1-t_{k+1}/t_k$ in \eqref{sobolev_iteriamola_00} one obtains for each $k$
\[
\|\gr\|_{t_k} \le \hat C (t_k - t_{k+1})^{- \frac{\nu-p}{p}} \|\gr\|_{t_{k+1}}^{\frac{1}{p}}.
\]
Iterating from 0 to $k$, this gives
\begin{equation}\label{eq:iterata_k}
\|\gr\|_t \leq \hat C^{1+1/p+\cdots+1/p^k}\cdot\Big[(t_0-t_1)(t_1-t_2)^{1/p}\cdots(t_k-t_{k+1})^{1/p^k}\Big]^{-\frac{\nu-p}p}\cdot\|\gr\|_{t_{k+1}}^{1/p^{k+1}}.\end{equation}
In view of the asymptotics of $\gr$ near $o$ (cf. Lemma \ref{teo_localsingular}), the last term is not hard to handle. Thus, for the constant $C_{p,\nu}$ in Conjecture \ref{teo_sobolev} to be bounded as $p \to 1$, we need the choice of $\{t_k\}$ to satisfy
\[\prod_{k\geq0}(t_k-t_{k+1})^{1/p^k}\geq Ae^{-\frac{B}{p-1}}\]
for some constants $A,B$ independent of $p$.

The gap in \cite{ajm} lies in the choice of $t_k$. More precisely, see \cite[pp.822-824]{ajm}: the sequence $\{\sigma_k\}$ constructed therein becomes negative for sufficiently large $k$, which is forbidden by the iteration scheme. In fact, this problem is subtle in view of the following result, which shows that \emph{no iteration} of \eqref{sobolev_iteriamola_00} can lead to a constant $C_{p,\nu}$ which is bounded as $p \to 1$:

\begin{proposition}
    For each $A, B, t_0>0$ there exists a constant $p_0=p_0(A,B,t_0)>1$ such that: for all $1<p\leq p_0$ and all sequences $\{x_k>0\}_{k\geq0}$ with $\sum_{k\geq0}x_k\leq t_0$, we have
    \[\prod_{k\geq0} x_k^{1/p^k}<Ae^{-\frac B{p-1}}.\]
\end{proposition}
\begin{proof}
    Increasing $t_0$ only strengthens the statement. Thus we increase $t_0$ so that $e^{B+2}t_0$ is an integer. Let $I=\big\{k\in\mathbb{N}: x_k>e^{-B-2}\big\}$, thus $\# I\leq e^{B+2}t_0+1$. We can estimate
    \[\begin{aligned}
        \sum_{k\geq0}\frac1{p^k}\log x_k &\leq (\# I)\log\max(1,t_0)-(B+2)\sum_{k\notin I}\frac1{p^k} \\
        &\leq \big(e^{B+2}t_0+1\big)\log\max(1,t_0)-(B+2)\sum_{k=e^{B+2}t_0+1}^\infty\frac1{p^k} \\
        &= \big(e^{B+2}t_0+1\big)\log\max(1,t_0)-(B+2)\frac{p^{-e^{B+2}t_0}}{p-1}.
    \end{aligned}\]
    Now we choose $p_0$ sufficiently close to 1, such that
    \[(B+1)p_0^{-e^{B+2}t_0}>B\ \ \text{and}\ \ \frac{p_0^{-e^{B+2}t_0}}{p_0-1}>\big(e^{B+2}t_0+1\big)\log\max(1,t_0)-\log A\]
    With this choice, for all $1<p\leq p_0$ we have
    \[\sum_{k\geq0}\frac1{p^k}\log x_k<\log A+\frac{B}{p-1}. \qedhere\]
\end{proof}

Notwithstanding the previous proposition, it is convenient to state the estimate for the Green kernel with the unbounded constant as $p\to 1^+$, given its potential use in applications beyond the present work.

\begin{theorem}\label{nonstable_sobolev_theorem}
Let $\Omega \subset M$ be a connected open set, and denote by $r$ the distance from a fixed origin $o \in \Omega$. Assume that $\Omega$ supports the weighted Sobolev inequality \eqref{sobolev_weighted} for some $p \in (1,\nu)$, constant $\So_{p,\nu}>0$ and weight $\eta$ satisfying \eqref{def_eta_sobolev}. Then, $\Delta_p$ is non-parabolic on $\Omega$ and, letting $\gr(x)$ be the Green kernel of $\Delta_p$ on $\Omega$ with pole at $o$,
\begin{equation}\label{nonstable_upper_consobolev}
\gr(x) \le C_{p,\nu}^{\frac{1}{p-1}} \eta\big(2r(x)\big)^{\frac{1}{p-1}} r(x)^{- \frac{\nu-p}{p-1}}, \qquad \forall \, x \in  \Omega \backslash \{o\},
\end{equation}
where
$$
C_{p,\nu}=\frac{\So_{p,\nu}^{\frac{\nu}{p}} \bar C_{p, \nu}^{\frac{\nu-p}{p}} (ep)^{\nu-p}}{(p-1)^{\nu-p}}
$$
and $\bar{C}_{p,\nu}$ is as in \eqref{Cpnu_mag0}. In particular, if
\begin{equation}\label{nonstable_ipo_eta_2}
\eta(t) = o \left( t^{\nu - p} \right) \qquad \text{as } \, t \ra \infty.
\end{equation}
then
 $\gr(x) \ra 0$ as $r(x) \ra \infty$ in $\Omega$.
\end{theorem}

\begin{proof}
    Choosing $t_{k}=t/p^k$, we have
    \begin{align*}
        \prod_{k\geq 0} (t_k - t_{k+1})^{1/p^k} &=\prod_{k\geq 0} \left(\frac{(p-1)t}{p^{k+1}}\right)^{1/p^k} = \left(\frac{p-1}{p}t\right)^{\sum_{k\geq 0} 1/p^k} p^{-\sum_{k\geq0}\frac{k}{p^k}}
        \\&=\left(\frac{p-1}{p}t\right)^{\frac{p}{p-1}} p^{-\frac{p}{(p-1)^2}}.
    \end{align*}
    Taking the limit as $k\to \infty$ in \eqref{eq:iterata_k}, we conclude
    \begin{equation*}
        \Vert \gr \Vert_t \leq \hat{C}^{\frac{p}{p-1}}\left( \frac{p-1}{p} t \right)^{-\frac{\nu-p}{p-1}} p^{\frac{(\nu-p)}{(p-1)^2}} \leq  \hat{C}^{\frac{p}{p-1}}\big(({p-1}) t \big)^{-\frac{\nu-p}{p-1}} (ep)^{\frac{\nu-p}{p-1}},
    \end{equation*}
    which because of the definition of $\hat{C}$ in \eqref{def_Chat} gives the desired bound.
\end{proof}

\subsection{New decay estimates}\label{subsec:new_est}

Our goal of this subsection is to prove the following decay estimate for the $p$-Green's kernel:

\begin{theorem}\label{theorem-decayGreen2}
    Let $M^n$ be a complete manifold satisfying $\dous$ and $\neuuno$, with doubling dimension $\nu=\log_2 C_\Dou$. Assume that there exist constants $C_{\RD}>0$, $b\in (1, \nu]$ such that
    \begin{equation}\label{eq_reverse_doub}
    \frac{|B_t(o)|}{|B_s(o)|} \ge C_{\RD} \left( \frac{t}{s}\right)^b\qquad\forall \, t \ge s > 0.
    \end{equation}
    Then, for each $p_0 \in (1,b)$ and $p\in(1,p_0]$, the Green kernel of $\Delta_p$ on $M$ with pole at $o \in M$ satisfies
    \begin{equation}\label{eq_decaygreen}
    \gr(x) \le C^{\frac{1}{p-1}} \left[ \sup_{t \in (0,2r(x))} \frac{t^\nu}{|B_t(o)|}\right]^{\frac{1}{p-1}} r(x)^{- \frac{\nu-p}{p-1}}, \qquad \forall \, x \in  M \backslash \{o\},
    \end{equation}
    where $C$ is a constant only depending on $C_\Dou$, $\Po_{1,1}$, $p_0$, $b$ and $C_{\RD}$.
\end{theorem}

This theorem recovers the full strength of \cite[Theorem 3.23]{ajm}, up to the fact that the constant $C$ here is not explicitly computed.

The proof of Theorem \ref{theorem-decayGreen2} consists of two steps. First, we use a capacity argument to bound $\min_{\p B_t(o)}\gr$ from above for all radii $t$. This argument is similar to the one employed in \cite[Theorem 3.11]{afnp}. Then, we use a Harnack inequality and a chaining argument to estimate $\max_{\p B_t(o)}\gr$ in terms of $\min_{\p B_t(o)}\gr$, thus obtaining the desired upper bound for $\gr$.

The following theorem realizes the first step:

\begin{theorem}\label{teo_sobolev_inf}
    Let $\Omega \subset M$ be a connected open subset of $M$, possibly $\Omega = M$, and denote by $r$ is the distance from a fixed origin $o \in \Omega$. Assume that $\Omega$ supports the weighted Sobolev inequality \eqref{sobolev_weighted} for some $p \in (1,\nu)$, constant $\So_{p,\nu}>0$ and weight $\eta$ satisfying \eqref{def_eta_sobolev}. Then, $\Delta_p$ is non-parabolic on $\Omega$ and, letting $\gr(x)$ be the Green kernel of $\Delta_p$ on $\Omega$ with pole at $o$, for each $t>0$ we have
    \begin{equation}\label{stable_upper_consobolev}
    \inf_{\partial B_t(o)} \gr \le \hat C_{p,\nu}^{\frac{1}{p-1}} \eta\big(2t\big)^{\frac{1}{p-1}} t^{- \frac{\nu-p}{p-1}},
    \end{equation}
    where
    $$
    \hat C_{p,\nu} \doteq \So^{\frac{\nu}{p}}_{p,\nu} \ 2^{\frac{(\nu+2p)(\nu-p)}{p}}
    $$
    is bounded as $p \ra 1$ if so is $\So_{p,\nu}$.
\end{theorem}

\begin{proof}
By taking a smooth exhaustion $\Omega_j \uparrow \Omega$ by connected domains, it is enough to prove the result for $\gr$ the kernel of a relatively compact, smooth open subset, still called $\Omega$. In particular, $\gr$ is smooth and vanishes on $\partial \Omega$. Write 
\[
f(t) \doteq \min_{\partial B_t} \gr.
\] 
If $f(t) =0$, there is nothing to prove. Otherwise, since $\Omega$ is connected, we have $B_t \Subset \Omega$ and, as a consequence of the strong maximum principle,
\[
B_t \subset \big\{ \gr > f(t) \big\}.
\]
By \eqref{eq_nice_pcapac} and the monotonicity of the $p$-capacity, it follows that
\[
\capac_p(B_t,\Omega) \le \capac_p\left(\{ \gr > f(t)\}, \Omega \right) = f(t)^{1-p}.
\]
Hence, 
\[
f(t) \le \Big( \capac_p(B_t,\Omega) \Big)^{-\frac{1}{p-1}}.
\]
Next, we observe that in view of the weighted Sobolev inequality, for each $\psi \in \lip_c(\Omega)$ with $\psi \ge 1$ on $B_t$ we have
\[
\int_\Omega |\nabla \psi|^p \ge \So_{p,\nu}^{-1} \left( \int_{\Omega} \eta(r)^{-\frac{p}{\nu-p}} |\psi|^{\frac{\nu p}{\nu - p}}\right)^{\frac{\nu-p}{\nu}} \ge \So_{p,\nu}^{-1} \eta(t)^{-\frac{p}{\nu}} |B_t|^{\frac{\nu-p}{\nu}}.
\]
Here, in the last inequality we restricted the integral to $B_t$ and used the monotonicity of $\eta$. By taking the infimum over $\psi$, 
\[
\capac_p(B_t,\Omega) \ge \So_{p,\nu}^{-1} \eta(t)^{-\frac{p}{\nu}} |B_t|^{\frac{\nu-p}{\nu}}.
\]
Thus, 
\[
f(t) \le \So_{p,\nu}^{\frac{1}{p-1}} \eta(t)^{\frac{p}{\nu(p-1)}} |B_t|^{-\frac{\nu-p}{\nu(p-1)}}
\]
Next, since for each $\psi \in \lip_c(\overline{B}_{t})$ vanishing on $\partial B_t$ the unweighted Sobolev inequality
\[
\eta(2t)^{-\frac{p}{\nu}} \left(\int_{B_{2t}} |\psi|^{\frac{\nu p}{\nu - p}}\right)^{\frac{\nu-p}{\nu}} \le 
\left(\int_{\Omega} \eta(r)^{-\frac{p}{\nu-p}} \psi^{\frac{\nu p}{\nu - p}}\right)^{\frac{\nu-p}{\nu}}  \le \So_{p,\nu} \int_\Omega |\nabla \psi|^p = \So_{p,\nu} \int_{B_{2t}} |\nabla \psi|^p
\]
holds, by the argument in \cite[Proposition 2.1]{akut} (see also \cite{carron} and \cite[Proposition 3.1]{pst}), there exists an explicit constant $\bar C_{p,\nu}$ depending on $\So_{p,\nu} \eta(2t)^{\frac{p}{\nu}}$ such that
\[
|B_s| \ge \bar C_{p,\nu} s^\nu \qquad \text{for each } \, s \in (0,t]
\]
By (see \cite[p.375 line 10]{pst}), the constant is explicitly given by 
\[
\bar C_{p,\nu} = 2^{- \frac{\nu(\nu + 2p)}{p}} \So_{p,\nu}^{-\frac{\nu}{p}} \eta(2t)^{-1}. 
\]
Indeed, following the notation in \cite{pst}, it holds
\[
q \doteq \frac{\nu p}{\nu-p}, \quad S_{p,q} \doteq \left[\So_{p,\nu} \eta(2t)^{\frac{p}{\nu}} \right]^{-\frac{1}{p}}, \quad \alpha \doteq \frac{\nu}{\nu+p}, \quad \bar \alpha \doteq \frac{\nu(\nu + p)}{p^2}.
\]
By using again the monotonicity of $\eta$, we obtain 
\[
f(t) \le \left(\bar C_{p,\nu}^{-\frac{\nu-p}{\nu}} \So_{p,\nu}\right)^{\frac{1}{p-1}} \eta(t)^{\frac{p}{\nu(p-1)}} t^{-\frac{\nu-p}{p-1}} \le \hat C_{p,\nu}^{\frac{1}{p-1}} \eta(2t)^{\frac{1}{p-1}} t^{-\frac{\nu-p}{p-1}}
\]
with
\[
\hat C_{p,\nu} \doteq \So^{\frac{\nu}{p}}_{p,\nu} 2^{\frac{(\nu+2p)(\nu-p)}{p}}. \qedhere
\]
\end{proof}

Next, we need to bound the ratio between the minimum and maximum of $\gr$ on $\p B_t(o)$. This can be done by finding a path connecting the extremal points of $\gr$, then covering this path with small balls, and finally using the Harnack inequality on each ball. Notice that the balls must avoid the pole $o$ when applying the Harnack inequality therein. Hence, we need a quantitative lower bound for the best radius $t'$ so that $\p B_t(o)$ is connected outside $B_{t'}(o)$.

\begin{definition}
    Let $M$ be a complete Riemannian manifold. For each fixed origin $o$, define the RCA (``relatively connected annuli") function 
    \[\begin{array}{c}
        \varsigma_o \ : \ (0,\infty) \to [1,\infty), \\[0.5cm]
	\disp \varsigma_o(t) = \inf \left\{ c \in [1,\infty) \ : \
    \begin{array}{l}
	\text{every $x,y \in \partial B_t(o)$ can be joined} \\
	\text{by a path lying in $\overline{B_t(o)}\backslash B_{t/c}(o)$}
	\end{array} \right\}
    \end{array}\]
\end{definition}
Observe that
\[
\varsigma_o(t) \le \frac{t}{{\rm inj}(o)}<\infty,
\]
where ${\rm inj}(o)$ is the injectivity radius of $o$. The terminology RCA was introduced by  V. Minerbe in \cite{minerbe}, where the author shows the following remarkable result: 

\begin{proposition}[\cite{minerbe}, Proposition 2.8]\label{teo_minerbe}
	Let $M$ be a connected, complete manifold satisfying $\dous, \neuuno$. Suppose further that there exists $C_{\RD}>0$, $b>1$ and a point $o \in M$ such that
	\begin{equation}\label{eq_reversevol_2}
		\forall t \ge s >0, \qquad \frac{|B_t(o)|}{|B_s(o)|} \ge C_{\RD} \left( \frac{t}{s}\right)^b.
	\end{equation}
	Then, there exists $\varsigma_0>1$ depending on $C_\Dou, \Po_{1,1}, C_{\RD}, b$ such that 
	\[
	\varsigma_o(t) \le \varsigma_0 \qquad \forall \, t \in (0,\infty).
	\]
\end{proposition}

\begin{remark}
	Because of $\dous, \neuuno$ and Remark \ref{rem_Pp} the constant $\varsigma_0$, explicitly computed in \cite{minerbe}, remains bounded as $p \ra 1$ while it diverges as $p \ra b$, see \cite[p. 1715]{minerbe}.
\end{remark}

A chaining argument shows the following:

\begin{proposition}\label{teo_sobolev_inf_2}
    Fix $p \in (1, \infty)$. Let $M$ be a complete Riemannian manifold, $\Omega \subset M$ be a smooth open subset and let $\gr$ be the Green kernel of $\Delta_p$ on $\Omega$ with pole at $o \in \Omega$. Let $R>0$ satisfy 
    \[
    B_{6R} \Subset \Omega,  
    \] 
    where $B_R$ denotes a ball centered at $o$. Assume that $B_{6R}$ satisfies: 
    \begin{itemize}
    	\item the doubling property $\dous$ with constant $C_\Dou$;
    	\item the weak $(1,p)$-Poincar\'e inequality $\neuunop$ with constant $\Po_{1,p}$;
    	\item the Sobolev inequality% \textcolor{blue}{[with $\nu=n$!]}
    	\begin{equation}\label{eq_Sobolev_supinf}
    	\left( \int |\psi|^{ \frac{\nu p}{\nu-p}}\right)^{\frac{\nu-p}{\nu}} \le \So_{p,\nu} \int |\nabla \psi|^p \qquad \forall \, \psi \in \lip_c(B_{6R}),
        \end{equation}
    	for some $\nu>p$ and $\So_{p,\nu}>0$. 
    \end{itemize}
    Then, having fixed $p_0 \in (p,\nu)$, the following Harnack inequality holds:
    $$\sup_{\partial B_t} \gr \le\Ha_{p,\nu}^{\frac{1}{p-1}} \cdot \inf_{\partial B_t} \gr \qquad \forall \, t \in (0,R),$$
    with constant
    \begin{equation}\label{ine_harnack_4}
    	\Ha_{p,\nu} = \exp\left\{C \varsigma_o(t)^{ \nu \left( 1 + \frac{2\nu}{\nu-p_0}\right)} \Po_{1,p} \max\{1, \So_{p,\nu}\}^{\frac{2\nu^2}{\nu-p_0}} 
        \right\},
    \end{equation}
    where $C>0$ is a constant depending on $C_\Dou$, $n, \nu$, $p_0$, $|B_1(o)|$ and on a lower bound on $\Ric$ on $B_{24}(o)$.
\end{proposition}

\begin{proof}
   For $t \in (0,R]$, define $\hat{t} = \frac{t}{24\varsigma_o(t)}$. Let us take a maximal collection of disjoint balls $\{B_{\hat t}(x_i)\}_{i=1}^N$ of radius $\hat t$ with centers in the annulus
    \[
        A_t(o) = \overline{B_t(o)} \backslash B_{12\hat t}(o). 
    \]
    Assume that $B_{\hat t}(x_1)$ is the ball with the smallest volume among those in the collection. Since the balls $B_{\hat t}(x_i)$ are disjoint in $\tilde A_t(o)=:\overline {B_{ t+\hat t}(o)}\backslash {B_{11\hat t}(o)}$ and their double ${B_{2\hat t} (x_i)}$ cover $A_t(o)$, by  \eqref{rel_low_vol},
    $$
    |\tilde A_t(o)| \ge \sum_{i} |B_{\hat t}(x_i)| \ge N |B_{\hat t}(x_1)| \ge N C_1 \left( \frac{\hat t}{2t}\right)^\nu |B_{2t}(x_1)| \ge N \frac{C_2}
    {\varsigma_o(t)^\nu} |\tilde A_t(o)|,
    $$
    where $C_1,C_2$ depend on $C_\Dou$ and $\nu = \log_2 C_\Dou$. Thus
    $$
    N \le N_t \doteq \frac{\varsigma_o(t)^\nu}{C_2}.
    $$
    Fix $x,y \in \partial B_t(o)$ realizing, respectively, the maximum and minimum of $\gr$ on $\partial B_t(o)$, and let $\gamma \subset A_t(o)$ be a path from $x$ to $y$. The existence of $\gamma$ is guaranteed by the very definition on $\varsigma_o(t)$. Since the family $\{B_{2\hat t}(x_i)\}$ covers $A_t(o)$, up to renaming we can cover $\gamma$ by a chain of at most $N_t$ balls $B_l = B_{2 \hat t}(x_l)$ with $B_l \cap B_{l+1} \neq \emptyset$. From $6B_l \subset B_{6R}(o)\backslash \{o\}$, $\gr$ is nonsingular on $6B_l$ and we can apply Theorem \ref{teo_harnack} to deduce 
    $$
    \sup_{B_l} \gr \le \haus_{p,\nu,l}^{\frac{1}{p-1}} \inf_{B_l} \gr \qquad \text{for each $B_l$},
    $$
    where 
    \begin{equation}\label{ine_harnack_l}
    	\Ha_{p,\nu,l} = \exp\left\{ c_2 \Po_{1,p} \left[\frac{|B_{6\hat t}(x_l)|}{|B_{2\hat t}(x_l)|}\right]^{\frac{1}{p}} Q_l^{-2} p \right\} \le \exp\left\{ C_2 \Po_{1,p} Q_l^{-2} \right\}
    \end{equation}
    where $c_2$ is a constant depending only on $\nu$ and $p_0$, while $C_2$ also depends on $C_\Dou$, 
    $$
    Q_l = \inf_{\tau \in [1, \frac{\nu}{\nu-p}]} \big( \So_{p,\nu} C_{p,\nu} \big)^{- \frac{\nu \tau}{p}} \hat t^{\nu \tau} |B_{2\hat t}(x_l)|^{-\tau}
    $$
    and $C_{p,\nu}$ is as in \eqref{Cpnu}. First suppose that $t \ge 1/2$. By the definition of $\hat t$ and by \eqref{rel_low_vol}, we have
    \begin{equation}\label{eq_prim_bou}
    \begin{array}{lcl}
    Q_l & \ge & \disp \inf_{\tau \in [1, \frac{\nu}{\nu-p}]} \big( \So_{p,\nu} C_{p,\nu} \big)^{- \frac{\nu \tau}{p}} t^{\nu \tau} (24\varsigma_o(t))^{-\nu \tau} |B_{2t}(o)|^{-\tau} \\[0.5cm]
    & \ge & \disp \inf_{\tau \in [1, \frac{\nu}{\nu-p}]} \big( \So_{p,\nu} C_{p,\nu} \big)^{- \frac{\nu \tau}{p}} (48\varsigma_o(t))^{-\nu \tau} C_3^\tau |B_{1}(o)|^{-\tau} \\[0.5cm]    
    & \ge & \disp C_4\left(\inf_{\tau \in [1, \frac{\nu}{\nu-p}]} \big( \So_{p,\nu} C_{p,\nu} \big)^{- \frac{\nu \tau}{p}}\right) \varsigma_o(t)^{-\frac{\nu^2}{\nu-p}} |B_{1}(o)|^{- \frac{\nu}{\nu-p}},
    \end{array}
    \end{equation}
    where $C_3$ depends on $C_\Dou$ and $C_4$ on $C_\Dou, \nu, p_0$. If $t\leq 1/2$, the Sobolev inequality in \eqref{eq_Sobolev_supinf} does not yield effective estimates for small $t$ because $\nu$ may be different from (actually, bigger than) $n$. For this reason, choose $\kappa$ so that $\Ric \ge -(n-1)\kappa^2$ on $B_{24}(o)$. Example \ref{ex_mrs_local} guarantees that the Sobolev inequality \eqref{eq_Sobolev_supinf} is satisfied on $B_{4}(o)$ with $\nu = n$ and constant $\So_{p,n}(1)$. In turn, since each ball $B_{8\hat t}(x_l) \subset B_{1}(o)$ we can apply the same procedure as above by using the Harnack inequality in Theorem \ref{teo_harnack} on each $B_{2\hat t}(x_l)$ coming from the choice of \eqref{SobolevL1_local} as the Sobolev inequality. The Harnack constant satisfies  \eqref{ine_harnack_l} with $\nu = n$, $c_2$ depending on $n,p_0$ and 
    $$
    Q = Q_l = \disp \inf_{\tau \in [1, \frac{n}{n-p}]} \big( \So_{p,n}(1) C_{p,n} \big)^{- \frac{n\tau}{p}} \hat t^{n\tau} |B_{2\hat t}(x_l)|^{-\tau}. 
    $$
    By Bishop-Gromov inequality, 
    \[
    \frac{|B_{2\hat t}(x_l)|}{t^n} \le C_\kappa \lim_{s \to 0} \frac{|B_{s}(x_l)|}{s^n} = C'_\kappa
    \]
    for constants $C_\kappa,C_\kappa' > 1$ depending on $\kappa, n$. Therefore, 
    \[
    Q_l \ge \disp \inf_{\tau \in [1, \frac{n}{n-p}]} \big( \So_{p,n}(1) C_{p,n} \big)^{- \frac{n\tau}{p}} (C_\kappa')^{-\frac{n}{n-p}}. 
    \]
    Using that $\So_{p,n}(1)$ only depends on $\kappa, |B_1(o)|$, the Harnack constant can be estimated by using the doubling condition as follows:
    \[
    \haus_{p,n,l} \le \exp \left\{ C_5 \Po_{1,p}\right\}
    \]
    where $C_5$ depends on $\nu, p_0,\kappa, n, |B_1(o)|$. Together with the bound on $\haus_{p,\nu,l}$ we obtained in \eqref{eq_prim_bou}, we get for each $t \in (0,R)$
    \[
    \sup_{B_l} \gr \le \hat\haus_{p,\nu,l}^{\frac{1}
    {p-1}} \inf_{B_l} \gr 
    \]
    with
    \[
        \begin{array}{lcl}
        \disp \hat \Ha_{p,\nu,l} & = & \disp \exp\Big\{C_6 \Po_{1,p} \sup_{\tau \in [1, \frac{\nu}{\nu-p}]} \left[ \max\{1,\So_{p,\nu}\}^{2\nu \tau/p} \varsigma_o(t)^{\frac{2 \nu^2}{\nu-p}} \right] \Big\} \\[0.5cm]
        & \le & \exp\left\{C_6 \Po_{1,p} \max\{1, \So_{p,\nu}\}^{\frac{2\nu^2}{\nu-p_0}} \varsigma_o(t)^{\frac{2 \nu^2}{\nu-p_0}} \right\}
        \end{array}
    \]
    and $C_6>0$ depending on $C_\Dou$, $n, \nu$, $p_0$, $|B_1(o)|$ and on a lower bound on $\Ric$ on $B_{24}(o)$. The result follows by chaining the Harnack inequalities on each $B_l$.
\end{proof}

If we assume that the entire $M$ satisfies $\dou, \neuuno$, then one may use the Harnack inequality in Theorem \ref{teo_uniHarnack} instead of Theorem \ref{teo_harnack} on each ball $B_l$ of the covering, obtaining the following:

\begin{corollary}\label{cor_sobolev_inf}
    Let $M$ be a complete manifold satisfying $\dou, \neuuno$, and assume that $\Delta_p$ is non-parabolic for each $p \in (1,p_0] \subset (1,\nu)$. Denoting by $\gr_p$ the Green kernel of $\Delta_p$ centered at $o \in M$, the following Harnack inequality holds:
    $$
    \sup_{\partial B_t} \gr \le\Ha_{p,\nu}^{\frac{1}{p-1}} \cdot \inf_{\partial B_t} \gr \qquad \forall \, t \in (0,\infty),
    $$
    with constant
    \begin{equation}\label{cor_ine_harnack_4}
    	\Ha_{p,\nu} = \exp\left(C \varsigma_o(t)^\nu\right),
    \end{equation}
    where $C>0$ is a constant depending on $C_\Dou, \Po_{1,1}$, $n, \nu$, $p_0$.
\end{corollary}

We can now prove Theorem \ref{theorem-decayGreen2}.

\begin{proof}[Proof of Theorem \ref{theorem-decayGreen2}]
    If $\dous, \neuuno$ and the reverse doubling inequality \eqref{eq_reverse_doub} hold, then applying Theorem \ref{theorem-tewo2} and Remark \ref{rem_Pp} the Sobolev inequality 
    \[
		\left(
		\int_M \left[ \frac{r^\nu}{|B_r|}
		\right]^{-\frac{p}{\nu-p} } |\psi|^{\frac{\nu p}{\nu-p}}
		\right)^{\frac{\nu-p}{\nu}}
		\leq
		\So_{p_0,\nu} \int_M |\nabla \psi|^p, \qquad \forall \,\psi\in \lip_c (M)
    \]
    is satisfied with a constant $\So_{p_0,\nu}$ depending on $C_\Dou, \Po_{1,1},\nu,b,C_\RD$ and $p_0$. Hence, $M$ satisfies \eqref{sobolev_weighted} with the choice $\eta(r)=\sup_{s\leq r}\big(\frac{s^\nu}{|B_s|}\big)$. Also, Proposition \ref{teo_minerbe} guarantees that $\varsigma_o(t) \le \varsigma_0$ for some constant $\varsigma_0$ depending on $C_\Dou, \Po_{1,1}, C_\RD$ and $b$. We can therefore apply Theorem \ref{teo_sobolev_inf} and Corollary \ref{cor_sobolev_inf} to obtain the desired conclusion. 
\end{proof}

We also have the following local result:
\begin{proposition}\label{prop_decaygreen_local}
    Let $M$ be a complete Riemannian manifold and $\Omega \subset M$ be a smooth open subset. Fix $p_0 \in (1,n)$ and assume that $\Delta_p$ is non-parabolic in $\Omega$ for some $p \in (1, p_0]$. Denote by $\gr$ the $p$-Green kernel of $\Omega$ with pole at $o$. Then, the following estimate holds for balls
    \[
    B_{6R} \Subset \Omega  
    \] 
    centered at $o$:
    \begin{equation}\label{eq:nondeg}
        \gr_p(x) \le C^{\frac{1}{p-1}} \left[ \|\gr_p\|_{L^\infty(\partial B_R)} +  r(x)^{-\frac{n-p}{p-1}}\right] \qquad \forall \, x \in B_R,
    \end{equation}
    where $C>0$ is a constant depending on $n,p_0$, $\|\varsigma_o\|_{L^\infty([0,R])}$, $R$, $|B_1|$, and on a lower bound on $\Ric$ on $B_{36R}$.
\end{proposition}
\begin{proof}
    Set $c_R = \|\gr_p\|_{L^\infty(\partial B_R)}$ and observe that $\gr_p - c_R$ is the $p$-Green kernel of $\{\gr_p > c_R\}$, a set contained in $B_R$ by the maximum principle. We thus apply Theorem \ref{teo_sobolev_inf} to $\gr_p - c_R$  and use Example \eqref{ex_mrs_local} to get
    \[\min_{\partial B_t} \gr_p \le c_R+C\big(n,p_0,R,K,|B_1|\big)^{\frac1{p-1}}t^{-\frac{\nu-p}{p-1}},\]
    where $K$ is the Ricci lower bound in $B_{36R}$ (here, we used  \eqref{SobolevL1_local} in Example \ref{ex_mrs_local} and volume comparison to bound the Sobolev constant in $B_{6R}$). We then apply Proposition \ref{teo_sobolev_inf_2} to $\gr_p$ and obtain the desired bound.
\end{proof}

\subsection{Summary of the results in \texorpdfstring{\cite{ajm}}{[10]}}\label{subsec:ajm}

As already mentioned, the gap in the proof of \cite[Theorem 3.6]{ajm} affects some of the results there. We now summarize their current status in light of the main theorems of the present paper:
\begin{enumerate}[label={(\roman*)}, topsep=1pt, itemsep=-0.3ex]
    \item Theorems 1.3 and 4.4 in \cite{ajm}. The claimed IMCF $w$ does indeed exist (equivalently, its re\-pa\-ra\-me\-tri\-zation $\varrho_1$ described in \cite{ajm} is positive on $M\backslash\{o\}$), and satisfies the claimed gradient estimate, but its properness is currently pending (i.e. the lower bound on $\varrho_1$ needs to be proved). The existence of $w$ can be shown as follows: combining the decay estimate for the infimum of $\gr_p$ established in our Theorem~\ref{teo_sobolev_inf}, and the gradient estimate in \cite[Theorem 2.19]{ajm}, one deduces that the reparametrized $p$-Green kernels $\varrho_p$ converge up to a subsequence to some $\varrho_1$ satisfying $\sup_{\partial B_t} \rho_1 > 0$ for each $t$. The positivity of $\varrho_1$ then follows from the strong maximum principle \cite[Theorem 4.2]{ajm}.\\ Theorems 1.3 and 4.4 in \cite{ajm} are superseded by our Theorem~\ref{thm-intro:main_euc_isop}, see also Question~\ref{qs-intro:p_convergence}. For its proof, see Section \ref{sec:growth} below.
    
    \item Theorem 1.4 holds as stated in \cite{ajm}. It is the present Theorem~\ref{thm-intro:main_ric_rd}.
    
    \item Theorem 3.23 in \cite{ajm} is fully recovered by our Theorem~\ref{theorem-decayGreen2}.
    
    \item Proposition 4.3 in \cite{ajm}. The result holds with minor modifications, which do not affect applications. First, in the statement one should explicitly assume that the function $h$ defining the model $M_h$ is increasing. This fact, implicitly used in the proof, is automatic in all applications in \cite{ajm}. 
    Next, instead of \cite[Theorem 3.6]{ajm} one may rely on Proposition \ref{prop_decaygreen_local}: the desired conclusion restricts to balls $B_{6R} \Subset \Omega$ and becomes
    \begin{equation}\label{eq_nondegen}
    \varrho_1(x) \ge \min \Big\{ v_h^{-1} \left(Cr(x)^{n-1}\right),\ v_h^{-1}\big(C \min_{\partial B_R}v_h(\varrho_1)\big) \Big\},\qquad\forall\,x\in B_R,
    \end{equation}
    where $C$ is a constant depending on $n,p_0$, $\|\varsigma_o\|_{L^\infty([0,R])}$, $R$, $|B_1(o)|$ and on a lower bound on $\Ric$ on $B_{36R}$. This inequality suffices for applications to \cite[Theorem 4.6]{ajm} (whose assumptions imply $\varsigma_o(t) \le \varsigma_0$).\\
    To see \eqref{eq_nondegen}, we recall that $\gr_p=\int_{\varrho_p}^\infty v_h^{-\frac1{p-1}}$ and $\varrho_p\to\varrho_1$. Then taking the $(p-1)$-th power of \eqref{eq:nondeg}, letting $p\to1$ and using \cite[Lemma 4.1]{ajm} we obtain
    \[\begin{aligned}
        \frac1{v_h(\varrho_1(x))} &= \lim_{p\to1}\gr_p(x)^{p-1}\leq C\max\Big\{\lim_{p\to1}\|\gr_p\|_{L^\infty(\p B_R)}^{p-1},\ r(x)^{1-n}\Big\} \\
        &= C\max\Big\{\max_{\p B_R}\frac1{v_h(\varrho_1)},r(x)^{1-n}\Big\}.
    \end{aligned}
    \]
    which yields \eqref{eq_nondegen} up to renaming $C$. If in addition $h$ is convex, as for instance in application to \cite[Theorem 4.6]{ajm}, then $v_h^{-1}$ is concave and thus replacing $C$ with $\min\{C,1\}$ we get 
    \[
    \varrho_1(x) \ge C \min \Big\{ \min_{\p B_R} \varrho_1, v_h^{-1}\left( r(x)^{n-1}\right) \Big\}.
    \]
    
    \item Theorem 4.6 in \cite{ajm}. This a generalization of Theorem 1.4 and depends on the estimates in \cite[Theorem 3.23 and Proposition 4.3]{ajm}. Hence, it holds as stated in \cite{ajm}.

    \item Proposition 5.2 holds as stated in \cite{ajm}. It is enough to replace the use of \cite[Theorem 3.6]{ajm} with Proposition \ref{prop_decaygreen_local}, and to require that $B_{6R_0}(o)$ remains within the domain $\mathscr{D}_o$ of a normal chart at $o$.
\end{enumerate}

\section{Growth estimates for IMCF cores}\label{sec:growth}

The goal of this section is to prove Theorems \ref{thm-intro:main_ric_rd} and \ref{thm-intro:main_euc_isop}. Let us recall that:
\begin{definition}\label{def-growth:imcf_core}
    We say that $u:M\backslash\{o\}\to\RR$ is an IMCF core with pole $o$, if:
    \begin{enumerate}[label={(\arabic*)}, nosep]
        \item $u\in\Lip_{\loc}(M\backslash\{o\})$, and $u(x)-(n-1)\log r(x)\to0$ as $x\to o$,
        \item $u$ is a weak IMCF in $M\backslash\{o\}$.
    \end{enumerate}
    \noindent We say that $u$ is a proper IMCF core, if $E_t(u)\Subset M$ for all $t\in\RR$.
\end{definition}

\subsection{Proof of Theorem \ref{thm-intro:main_ric_rd}}

We restate Theorem \ref{thm-intro:main_ric_rd} below for the reader's convenience:

\begin{theorem}\label{thm-pcore:p_core}
    Suppose $M$ satisfies $\Ric\geq0$ and a reverse volume doubling condition
    \begin{equation}\label{eq-pcore:rvdoub}
        \frac{|B_t(o)|}{|B_s(o)|} \ge C_{\RD} \left( \frac{t}{s}\right)^b,\qquad\forall\,t\geq s>0. \tag{$\text{RD}_b$}
    \end{equation}
    Then there exists a unique proper IMCF core $u$ on $M$ with pole $o$, with the growth estimate
    \begin{align}
        & u(x)\geq(n-1)\log r(x)-C\big(n,b,C_{\RD},|B_1(o)|\big),\qquad \forall\,x\in B_1(o), \label{eq-pcore:growth_ric_rd_1}\\
        & u(x)\geq(b-1)\log r(x)-C\big(n,b,C_{\RD},|B_1(o)|\big),\qquad \forall\,x\in M\backslash B_1(o), \label{eq-pcore:growth_ric_rd_2}
    \end{align}
    and the global gradient estimate
    \begin{equation}\label{eq-pcore:grad_est}
        |\D u(x)|\leq(n-1)e^{-\frac{u}{n-1}},\qquad\forall\,x\in M\backslash\{o\}.
    \end{equation}
\end{theorem}
\begin{proof}
    Recall Definition \ref{def-prelim:vd_poincare}: it is well-known that $\Ric\geq0$ implies $\dous$ and $\neuuno$, with the doubling dimension $\nu=n$ and constants $C_{\Dou},\Po_{1,1}$ bounded in terms of $n$. Further, notice that
    \[\frac{t^\nu}{|B_t(o)|}\leq\frac1{|B_1(o)|},\qquad\forall\,t\leq1\]
    by Bishop-Gromov, and
    \[\frac{t^\nu}{|B_t(o)|}\leq\frac{t^{n-b}}{C_{\RD}|B_1(o)|},\qquad\forall\,t\geq1\]
    by our reverse doubling condition.
    Thus, Theorem \ref{theorem-decayGreen2} (applied with the choice $p_0=(b+1)/2$) yields
    \begin{equation}\label{eq-pcore:gr_lower}
        \gr_p(x)\leq\left\{\begin{aligned}
        & C^{\frac1{p-1}}|B_1(o)|^{\frac1{1-p}}r(x)^{-\frac{n-p}{p-1}},\qquad\forall\,x\in B_1(o)\backslash\{o\}, \\
        & C^{\frac1{p-1}}|B_1(o)|^{\frac1{1-p}}r(x)^{-\frac{b-p}{p-1}},\qquad\forall\,x\in M\backslash B_1(o),
    \end{aligned}\right.
    \end{equation}
    for all $p\in(1,(1+b)/2]$, where $C$ depends on $n,b,C_{\RD}$.

    Take Moser's transformation $w_p=(1-p)\log\gr_p$. The above bound is turned into
    \begin{align}
        w_p(x) &\geq (n-p)\log r(x)-C\qquad\forall\,x\in B_1(o)\backslash\{o\}, \label{eq-pcore:lower1}\\
        w_p(x) &\geq (b-p)\log r(x)-C\qquad\forall\,x\in M\backslash B_1(o), \label{eq-pcore:lower2}
    \end{align}
    where $C$ is a generic constant depending on $n,b,C_{\RD},|B_1(o)|$. Furthermore, the gradient estimate in \cite[Lemma 2.17]{ajm} gives
    \begin{equation}\label{eq-pcore:grad_est_p}
        |\D w_p|\leq(n-p)e^{-\frac{w_p}{n-p}}\Big(\frac{n-p}{p-1}\omega_{n-1}^{\frac1{p-1}}\Big)^{\frac{p-1}{n-p}}\qquad\text{in}\ \ M\backslash\{o\}.
    \end{equation}
    To pass to the limit, we also need an upper bound for $w_p$. Suppose that $C$ is large enough so that the injectivity radius of $o$ is larger than $1/C$, and the mean curvature $H_\rho$ of the geodesic sphere $\p B_\rho(o)$ satisfies
    \begin{equation}\label{eq-pcore:Hrho}
        \Big|H_\rho-\frac{n-1}\rho\Big|\leq1,\qquad\forall\,\rho\leq1/C.
    \end{equation}
    Consider the function
    \[\bar w_p=\bar w_p(r)=(n-p)\log r+\frac1{1-Cr},\qquad\forall\,r\leq1/C.\]
    Differentiating,
    \[\D\bar w_p=\Big[\frac{n-p}r+\frac{C}{(1-Cr)^2}\Big]\p_r,\]
    and
    \[\begin{aligned}
        \Delta_p\bar w_p-|\D\bar w_p|^p &= |\D\bar w_p|^{p-2}\Big[(p-1)\Big(\!-\frac{n-p}{r^2}+\frac{2C^2}{(1-Cr)^3}\Big)+H_r\Big(\frac{n-p}r+\frac{C}{(1-Cr)^2}\Big) \\
        &\hspace{80pt} -\Big(\frac{n-p}r+\frac{C}{(1-Cr)^2}\Big)^2\Big] \\
        &\hspace{-36pt}\leq |\D\bar w_p|^{p-2}\Big[\frac{2(p-1)C^2}{(1-Cr)^3}+\frac{(n-1)C+(n-p)(1-Cr)^2+Cr-2(n-p)C}{r(1-Cr)^2} \\
        &\hspace{36pt} -\frac{C^2}{(1-Cr)^4}\Big],
    \end{aligned}\]
    where we inserted \eqref{eq-pcore:Hrho} into the expression. When $C$ is sufficiently large and $p$ is sufficiently close to $1$, the numerator of the second term is bounded above by $-C/2$. Hence we have
    \[\frac{\Delta_p\bar w_p-|\D\bar w_p|^p}{|\D\bar w_p|^{p-2}}\leq\frac{2(p-1)C^2}{(1-Cr)^3}-\frac{C/2}{r(1-Cr)^2}-\frac{C^2}{(1-Cr)^4},\]
    and the right hand side is negative by Young's inequality. So by possibly increasing $C$, $\bar w_p$ is a supersolution of the $p$-IMCF equation, and hence $e^{-\bar w_p/(p-1)}$ is $p$-subharmonic in $B_{1/C}(o)$. By the comparison principle, we have
    \begin{equation}\label{eq-pcore:sharp_upper}
        w_p\leq\bar w_p-1+\log\Big[\omega_{n-1}\Big(\frac{p-1}{n-p}\Big)^{1-p}\Big]\qquad\text{in}\ \ B_{1/C}(o)\backslash\{o\}.
    \end{equation}
    Therefore, for some possibly larger generic constant $C$, we have
    \begin{equation}\label{eq-pcore:upper1}
        w_p\leq(n-p)\log r+C\qquad\text{in}\ \ B_{1/C}(o)\backslash\{o\}.
    \end{equation}

    Now, combining the two-sided bounds \eqref{eq-pcore:lower1} \eqref{eq-pcore:upper1} and the gradient estimate \eqref{eq-pcore:grad_est_p}, and Arzela-Ascoli's theorem, it follows that for some sequence $p\to1$ we can extract a limit
    \[w_p\to w\qquad\text{in}\ \ C^0_{\loc}(M\backslash\{o\}),\]
    for some solution $w$ of the weak IMCF in $M\backslash\{o\}$. The $C^0$ bounds \eqref{eq-pcore:lower1} \eqref{eq-pcore:lower2} \eqref{eq-pcore:upper1} imply that $w\geq(b-1)\log r-C$ outside $B_1(o)$ and
    \begin{equation}\label{eq-pcore:bound_for_w}
        (n-1)\log r-C\leq w\leq(n-1)\log r+C\qquad\text{in}\ \ B_{1/C}(o).
    \end{equation}

    It remains to show that 
    \[
    w = (n-1)\log r+\log\omega_{n-1} + o_r(1)
    \]
    as $x \to o$. Once this is proved, it follows that $\tilde w=w-\log\omega_{n-1}$ is an IMCF core with pole $o$, and satisfies the lower bounds \eqref{eq-pcore:lower1} \eqref{eq-pcore:lower2} together with the gradient estimate $|\D\tilde w|\leq(n-1)e^{-\tilde w/(n-1)}$ due to \eqref{eq-pcore:grad_est_p}. Thus, the proof of Theorem \ref{thm-pcore:p_core} is complete.

    We first claim that $w$ is approximately constant on $\p B_\rho(o)$ as $\rho\to0$. If this is not the case, then there is a sequence $\rho_i\to0$ and points $x_i,y_i\in\p B_{\rho_i}(o)$, such that $|w(x_i)-w(y_i)|\geq\eps$ for a uniform $\eps>0$. Then consider the shifted sequence $w_i=w-(n-1)\log\rho_i$ and the rescaled metric $g'_i=\rho_i^{-2}g$. Due to \eqref{eq-pcore:bound_for_w}, the functions $w_i$ are uniformly bounded in $B'_R(o)\backslash B'_{1/R}(o)$ for all $R>0$. Together with the gradient estimate, some subsequence of $w_i$ converges in $C^0_{\loc}(\mathbb R^n\backslash\{0\})$ to a weak IMCF $w_\infty$ with
    \[(n-1)\log|x|-C\leq w_\infty\leq(n-1)\log|x|+C,\qquad\forall\,x\in\R^n\backslash\{0\}.\]
    By Huisken-Ilmanen's Liouville theorem \cite[Proposition 7.2]{huiskenilmanen}, it follows that $w_\infty=(n-1)\log|x|+C'$ for some constant $C'$. On the other hand, we may take a further subsequence so that $x_i\to x_\infty\in\p B_1(0)$, $y_i\to y_\infty\in\p B_1(0)$. Hence, $\eps\leq\vert w_i(x_i)-w_i(y_i)|\to |w_\infty(x_\infty)-w_\infty(y_\infty)|=0$. This is a contradiction, and the claim follows.
    
    Thus, we may write $w=f(r)+\eta$ where $f:(0,1/C)\to\RR$ is some function in $r$, and $\eta\to0$ uniformly as $r\to0$. By further increasing $C$, we may assume $|\eta|<C$.
    By \eqref{eq-pcore:bound_for_w}, for all $s<1/C$ we have
    \begin{equation}\label{eq-pcore:aux1}
        f(s)\leq(n-1)\log(s)+C+\max_{\p B_{s}(o)}|\eta|\leq(n-1)\log(s)+2C.
    \end{equation}
    Now we notice that, for all $s<1/C$, we have
    \begin{equation}\label{eq-pcore:ineq_for_f}
        w\geq f(s)-\max_{\p B_s(o)}|\eta|+(n-1)\log\frac rs-3C^2r\qquad\text{on}\ \ \p B_s(o)\cup\p B_{1/C}(o).
    \end{equation}
    This is obvious on $\p B_s(o)$, and follows from \eqref{eq-pcore:bound_for_w} \eqref{eq-pcore:aux1} on $\p B_{1/C}(o)$. Since the right hand side of \eqref{eq-pcore:ineq_for_f} is a subsolution of IMCF in $B_{1/C}(o)\backslash B_s(o)$, by the maximum principle, we have
    \[w\geq f(s)-\max_{\p B_s(o)}|\eta|+(n-1)\log\frac rs-3C^2r\qquad\text{in}\ \ B_{1/C}(o)\backslash B_s(o).\]
    Similarly, one has
    \[w\leq f(s)+\max_{\p B_s(o)}|\eta|+(n-1)\log\frac rs+3C^2r\qquad\text{in}\ \ B_{1/C}(o)\backslash B_s(o).\]
    From these inequalities it follows that
    \[f(s)-(n-1)\log s-o_r(1)\leq f(r)-(n-1)\log r\leq f(s)+(n-1)\log s+o_r(1),\]
    for all $0<s<r\ll1/C$, where $o_r(1)$ is an expression that converges to $0$ as $r\to0$. This implies that $f(r)-(n-1)\log r$, hence $w-(n-1)\log r$ as well, uniformly converges to a constant when $r\to0$.

    Denote by $\alpha$ the constant such that
    \begin{equation}\label{eq-pcore:asymp_limit}
        w=(n-1)\log r+\alpha+o_r(1).
    \end{equation}
    It follows from \eqref{eq-pcore:sharp_upper} that $\alpha\leq\log\omega_{n-1}$.

    Let us show that in fact $\alpha=\log\omega_{n-1}$. We argue similarly as in \cite[Theorem 5.2]{ajm}. Take a small radius $R_0>0$ and a constant $K>0$, so that the sectional curvature is at most $K$ in $B_{R_0}(o)$ and $R_0<\pi/2\sqrt K$. Let $\bar\gr_p$ be the Green's kernel in the space form of curvature $K$. It can be explicitly computed that
    \begin{equation}\label{eq-pcore:bar_gr}
        \bar\gr_p(r)=\int_r^{R_0}v_K(s)^{-\frac1{p-1}}\,ds,
    \end{equation}
    where $v_K(s)$ is the volume of the sphere of radius $s$ in the model space. By comparison, 
    \[\gr_p\leq\bar\gr_p(r)+\sup_{\p B_{R_0}(o)}\gr_p.\]
    Combining this with the upper bound \eqref{eq-pcore:gr_lower}, it follows that
    \[\gr_p\leq\bar\gr_p(r)+C^{\frac1{p-1}}R_0^{-\frac{n-p}{p-1}}.\]
    Renormalizing and then taking $p\to1$, we obtain
    \[\begin{aligned}
        e^{-w}=\lim_{p\to1^+}e^{-w_p} &\leq \lim_{p\to1^+}\Big[\bar\gr_p(r)+C^{\frac1{p-1}}R_0^{-\frac{n-p}{p-1}}\Big]^{p-1} \\
        &=\max\Big\{\lim_{p\to1^+}\bar\gr_p(r)^{p-1},CR_0^{p-n}\Big\} \\
        &\leq \max\Big\{\frac1{v_K(r)},CR_0^{p-n}\Big\},
    \end{aligned}\]
    where in the last line we have used \cite[Lemma 4.1]{ajm}. Combining this with the asymptotic $w=(n-1)\log r+\alpha+o_r(1)$ and $v_K(r)\sim\omega_{n-1}r^{n-1}$, it follows that $\alpha\geq\log\omega_{n-1}$.

    Hence, we obtain the sharp asymptotics
    \[\lim_{p\to1^+}w_p=w=(n-1)\log r+\log\omega_{n-1}+o_r(1),\]
    as claimed. Combining this with \eqref{eq-pcore:grad_est_p}, we obtain the main gradient estimate \eqref{eq-pcore:grad_est}.
\end{proof}

\subsection{Proof of Theorem \ref{thm-intro:main_euc_isop}}

We restate Theorem \ref{thm-intro:main_euc_isop} for the reader's convenience:

\begin{theorem}\label{thm-growth:proper_core}
    Suppose $M$ satisfies the Euclidean isoperimetric inequality
    \begin{equation}\label{eq-growth:euc_isop}
        \Ps{E}\geq c_I|E|^{\frac{n-1}n},\qquad\forall\ E\Subset M,
    \end{equation}
    Then there exists a unique proper IMCF core $u$ with pole $o$, with the growth estimate
    \begin{equation}\label{eq-growth:growth_final}
	u(x)\geq (n-1)\log r(x)-C(n,c_I),\qquad\forall\,x\in M\backslash\{o\}.
    \end{equation}
\end{theorem}

The idea for proof is as follows. For each small $\rho>0$, we let $u_\rho$ be the unique proper weak IMCF on $M$ that starts from $B_\rho(o)$. The existence of such $u_\rho$ follows from \eqref{eq-growth:euc_isop} and \cite{Xu_2023_proper}. Then, we argue that $\big(u_\rho+(n-1)\log \rho\big)$ converges in $C^0_{\loc}(M\backslash\{o\})$ to the desired IMCF core as $\rho\to0$.

The key step in this proof is to obtain a lower bound of $u_\rho$ that is uniform when $\rho\to0$. The previous bound in \cite{Xu_2023_proper} is not enough for this purpose: Taking $E_0=B_\rho(o)$ into \cite[Theorem 4.1]{Xu_2023_proper}, the following bound was proved:
\begin{equation}\label{eq-intro:non_sharp}
    u(x)\geq\frac{n-1}n\log\big[r(x)-\rho\big]-\frac{n-1}n\log\rho-C(n,c_I),\qquad\forall\,x\in M\backslash B_\rho(o).
\end{equation}
However, the leading coefficient of $\log r(x)$ is nonoptimal, and the constant term diverges as $r\to0$ after the $+(n-1)\log\rho$ renormalization. The strategy here is to combine the diameter estimate technique in \cite[Lemma 4.4]{Xu_2023_proper} with a soft blow-up argument (see Claim 3 in the proof below), which gives a uniform lower bound of $u_\rho$.

The following two technical lemmas will be used in showing the asymptotics near $o$. They are not essential in showing the growth estimate \eqref{eq-growth:growth_final}.

\begin{lemma}\label{lemma-growth:small_ball_out_min}
    Suppose $M$ satisfies \eqref{eq-growth:euc_isop}. Then for each $o\in M$, there exists a sufficiently small radius $\rho_0$, so that $B_\rho(o)$ is outward perimeter-minimizing for all $\rho\leq\rho_0$.
\end{lemma}
\begin{proof}
    Suppose $\rho\leq\rho_0$. By \cite[Corollary 3.15]{Xu_2023_proper}, $B_\rho(o)$ admits a strictly outward area-minimizing hull $E\Subset M$. See \cite[Theorem 3.13]{Xu_2023_proper} for precise the definition of ``strictly outward area-minimizing hull''; here all we need are the following two properties:
    
    (i) if $B_\rho(o)$ is not outer area-minimizing, then $E\supsetneq B_\rho(o)$;
    
    (ii) $\p E\backslash\p B_\rho(o)$ is a minimal surface (rigorously saying, the support of a stationary integral varifold).
    
    Inserting our main condition \eqref{eq-growth:euc_isop} into \cite[(3.8)]{Xu_2023_proper}, we obtain a precise bound for the size of $E$: it holds $E\subset B_R(o)$ with
    \[R=\rho+2nc_I^{-\frac n{n-1}}\P{B_\rho(o)}^{\frac1{n-1}}.\]
    Here we note that the ``$+1$'' term in \cite[(3.8)]{Xu_2023_proper} is redundant, which is easily seen from the proof. Note that $R\to0$ when $\rho\to0$. Therefore, we may choose a small enough $\rho_0$ so that all the geodesic balls $B_{R'}(o)$, $R'\leq R$, are strictly mean convex. Then the maximum principle and fact (ii) above forces $E=B_\rho(o)$. The lemma then follows from fact (i) stated above.
\end{proof}

\begin{lemma}\label{lemma:subsol_extension}
    Let $\Omega\subset M$ be a $C^2$ domain, and $u\in\Lip_{\loc}(M)$. If the following hold:
    
    (i) $\Omega$ is locally outward perimeter-minimizing in $M$,

    (ii) $\Omega=E_T(u)$ and $u|_{M\backslash\Omega}\equiv T$ for some $T\in\RR$,
    
    (iii) $u|_\Omega$ is a subsolution of IMCF in $\Omega$,
    
    \noindent then $u$ is a subsolution of IMCF in $M$.
\end{lemma}
\begin{proof}
    Our goal is to check that for each $t\in\RR$, every competitor set $E\supset E_t$, and domain $K$ satisfying $E\backslash E_t\Subset K\Subset M$, it holds $J_u^K(E)\geq J_u^K(E_t)$. Slightly modifying $K$, we may assume that $\H^{n-1}\big(\p K\cap\p\Omega\big)=0$.
    
    Note that if $t>T$, then $E_t=M$, hence $E=M$ and the desired statement is vacuous. For $t=T$, we have $E_T=\Omega$, thus the result is a trivial combination of
    \[\Ps{E;K}\geq\Ps{\Omega;K}\qquad\text{and}\qquad\int_{E\cap K}|\D u|=\int_{\Omega\cap K}|\D u|.\]
    Now we assume $t<T$, hence $E_t\subset\Omega$ and $\p E_t\cap\p\Omega=\emptyset$. Note that $E\backslash\Omega\subset E\backslash E_t\Subset K$. Hence, in the minimizing property of $\Omega$, we can make the valid comparison $\Ps{\Omega;K}\leq\Ps{E\cup\Omega;K}$, which implies $\Ps{E\cap\Omega;K}\leq\Ps{E;K}$. Moreover, $\D u\equiv0$ in $E\backslash\Omega$. Hence
    \[J_u^K(E)=\Ps{E;K}-\int_{E\cap K}|\D u|\geq\Ps{E\cap\Omega;K}-\int_{(E\cap\Omega)\cap K}|\D u|.\]
    Setting $E'=E\cap\Omega$, it remains to show that
    \begin{equation}\label{eq-growth:subsol_ext_aux1}
        \Ps{E';K}-\int_{E'\cap K}|\D u|\geq\Ps{E_t;K}-\int_{E_t\cap K}|\D u|.
    \end{equation}
    Note that this does not directly follow from condition (iii), since $E'\backslash E_t$ is not precompact in $\Omega$. To show this, we consider the inner approximation
    \[\Omega_\epsilon=\big\{d(\cdot,M\backslash\Omega)>\epsilon\big\},\quad E'_\epsilon=E'\cap\Omega_\epsilon,\quad K_\epsilon=K\cap\Omega_\epsilon.\]
    Since $E'\backslash E_t\Subset M$ and $\p E_t\cap\p\Omega=\emptyset$, one finds that $E'_\epsilon\supset E_t$ for small enough $\epsilon$. Hence we can use (iii) to make the comparison
    \[\P{E'_\epsilon;K_{\epsilon/2}}-\int_{E'_\epsilon\cap K_{\epsilon/2}}|\D u|\geq \P{E_t,K_{\epsilon/2}}-\int_{E_t\cap K_{\epsilon/2}}|\D u|.\]
    Finally, we take $\epsilon\to0$ in this inequality. By a technical lemma \cite[Lemma 2.23]{Xu_2024_obstacle} we have $\P{E'_\epsilon,K_{\epsilon/2}}=\P{E'_\epsilon;K}\to\P{E';K}$. Further, note that $\Ps{E_t;K_{\epsilon/2}}=\Ps{E_t;K}$ for all small $\epsilon$ (since $\p E_t\cap\p\Omega=\emptyset$). Hence we obtain \eqref{eq-growth:subsol_ext_aux1} in the limit.
\end{proof}

\begin{proof}[Proof of Theorem \ref{thm-growth:proper_core}] {\ }
    
    Let $\rho_0<10^{-n}$ be as in Lemma \ref{lemma-growth:small_ball_out_min}. Further decreasing $\rho_0$, we may assume that
    \begin{equation}\label{eq-growth:small_spheres}
        \P{B_\rho(o)}\leq 2\omega_{n-1}\rho^{n-1},\qquad\forall\,\rho\leq\rho_0,
    \end{equation}
    and the mean curvature $H_\rho$ of $B_\rho(o)$ satisfies
    \begin{equation}\label{eq-growth:Hr}
        \big|H_\rho-(n-1)\rho^{-1}\big|<1,\qquad\forall\,\rho\leq\rho_0.
    \end{equation}
    For each $\rho\leq\rho_0$, let $\tilde u_\rho$ be the unique proper solution of IMCF in $M$ with initial value $B_\rho(o)$, as given by \cite[Theorem 1.2]{Xu_2023_proper}. Then set
    \[u_\rho=\tilde u_\rho+(n-1)\log\rho.\]

    We first establish a precise asymptotic of $u_\rho$ near $o$. Recall our notation $r=d(\cdot,o)$. Set
    \[\uu=(n-1)\log r-r,\qquad \ou=(n-1)\log r+r+\frac1{1-r/\rho_0}-1.\]
    Due to \eqref{eq-growth:Hr}, it is easy to verify
    \[
    \diver \left(\frac{\D\uu}{|\D\uu|} \right)>|\D\uu|\qquad\text{and}\qquad 
    \diver \left(\frac{\D\ou}{|\D\ou|}\right)<|\D\ou|\qquad\text{in}\ \ B_{\rho_0}(o)\backslash\{o\},\]
    namely, $\uu,\ou$ are subsolution and supersolution of IMCF in $B_{\rho_0}(o)\backslash\{o\}$.
    
    \vspace{6pt}
    \noindent\textbf{Claim 1.} $u_\rho\geq\uu$ in $B_{\rho_0}(o)\backslash B_\rho(o)$ and $u_\rho\geq(n-1)\log\rho_0-\rho_0$ on $M\backslash B_{\rho_0}(o)$.
    
    \vspace{3pt}
    \noindent\textit{Proof.} Extend $\uu$ by the constant value $(n-1)\log\rho_0-\rho_0$ on $M\backslash B_{\rho_0}(o)$. Combining Lemma \ref{lemma-growth:small_ball_out_min} and \ref{lemma:subsol_extension} (with the choice $\Omega=B_{\rho_0}(o)\backslash\{o\}$), $\uu$ is a subsolution of IMCF in $M\backslash\{o\}$. Moreover, we have $u_\rho>\uu$ in $\p B_\rho(o)$. Thus by the maximum principle \cite[Theorem 2.2(i)]{huiskenilmanen}, we have $u_\rho\geq\uu$ on $M\backslash B_\rho(o)$. The claim follows.\ \ $\Box$
    
    \vspace{6pt}
    \noindent\textbf{Claim 2.} $u_\rho\leq\ou$ in $B_{\rho_0}(o)\backslash B_\rho(o)$.
    
    \vspace{3pt}
    \noindent\textit{Proof.} Apply \cite[Theorem 2.2]{huiskenilmanen} inside $B_{\rho_0}(o)\backslash B_\rho(o)$, noting that $u_\rho<\ou$ in $\p B_\rho(o)$ hence $\{u_\rho>\ou\}\Subset B_{\rho_0}(o)\backslash B_\rho(o)$.\ \ $\Box$
    
    \vspace{3pt}
    Next, we prove the key growth estimate for $u_\rho$. For each $\rho\leq\rho_0$ and $t\geq(n-1)\log\rho$, define
    \[D(\rho,t):=e^{-\frac t{n-1}}\sup_{E_t(u_\rho)}(r).\]
    %\[D(r,t)=\omega_{n-1}^{1/(n-1)}\cdot\frac{\sup_{\p E_t(u_r)}\big(d(\cdot,o)\big)}{\P{E_t(u_r)}^{1/(n-1)}}.\]
    This quantity roughly measures the diameter of $E_t(u_\rho)$ compared to its area. Set
    \begin{equation}\label{eq-growth:choice_of_C}
        C_1=1+c_I^{-\frac n{n-1}}n(2\omega_{n-1})^{\frac1{n-1}}\frac{1+e^{n-1}}{1-e^{-1}}.
    \end{equation}
    
    \vspace{3pt}
    \noindent\textbf{Claim 3.} $D(\rho,t)\leq C_1$ for all $\rho<\rho_0/10$ and $t\geq(n-1)\log\rho$. As a result, we have
    \begin{equation}\label{eq-growth:lower_bound}
        u_\rho(x)\geq(n-1)\log\big[C_1^{-1}r(x)\big],\qquad\,x\in M\backslash B_\rho(o).
    \end{equation}
    
    \vspace{3pt}
    \noindent\textit{Proof.} Fix $\rho<\rho_0/10$.
    When $t\leq(n-1)\log\rho+n-1$, we may apply the argument in \cite[Lemma 4.4]{Xu_2023_proper} to obtain
    \[\begin{aligned}
        E_t(u_\rho) &\subset B\Big(o,\rho+\big(1+e^{t-(n-1)\log\rho}\big)nc_I^{-\frac{n}{n-1}}e^{\frac{t-(n-1)\log\rho}{n-1}}\P{B_\rho(o)}^{\frac1{n-1}}\Big) \\
        &\subset B\Big(o,\rho+\big(1+e^{n-1}\big)ne^{\frac{t}{n-1}}c_I^{-\frac{n}{n-1}}(2\omega_{n-1})^{\frac1{n-1}}\Big)\qquad\text{(by \eqref{eq-growth:small_spheres})}.
    \end{aligned}\]
    Hence $D(\rho,t)\leq 1+(1+e^{n-1})nc_I^{-\frac{n}{n-1}}(2\omega_{n-1})^{\frac1{n-1}}\leq C_1$.
    
    Now suppose that Claim 3 is not true. Then there exists $T>(n-1)\log\rho+n-1$ so that
    \begin{equation}\label{eq-growth:to_contradict}
        D(\rho,T)>C_1\qquad\text{and}\qquad D(\rho,T-n+1)\leq C_1,
    \end{equation}
    By the definition of $D(\rho,t)$, \eqref{eq-growth:to_contradict} means that
    \begin{equation}\label{eq-growth:cond_T_1}
        E_{T-n+1}(u_\rho)\subset B\big(o,C_1e^{T/(n-1)-1}\big)
    \end{equation}
    while
    \begin{equation}\label{eq-growth:cond_T_2}
	\big|E_T(u_\rho)\backslash B\big(o,C_1e^{T/(n-1)}\big)\big|>0.
    \end{equation}
    
    We now perform an argument which is similar to \cite[Lemma 4.4]{Xu_2023_proper}. For almost every $s\in\big[C_1e^{T/(n-1)-1},C_1e^{T/(n-1)}\big]$, define the quantities
    \[A(s)=\P{E_T(u_\rho);M\backslash\overline{B_s(o)}},\quad S(s)=\P{B_s(o);E_T(u_\rho)}\]
    and
    \[V(s)=\big|E_T(u_\rho)\backslash B_s(o)\big|.\]
    Using the excess inequality \cite[Lemma 2.21]{Xu_2024_obstacle}, see also \cite[(4.10)\,$\sim$\,(4.13)]{Xu_2023_proper}, we have
    \[\begin{aligned}
        \P{E_T(u_\rho)} &\leq \P{E_T(u_\rho)\cap B_s(o)} \\
        &\qquad +\Big[\!\exp\big(T-\inf_{E_T(u_\rho)\backslash B_s(o)}(u_\rho)\big)-1\Big]\P{B_s(o);E_T(u_\rho)}.
    \end{aligned}\]
    By \eqref{eq-growth:cond_T_1}, we have $\inf_{E_T(u_\rho)\backslash B_s(o)}(u_r)\geq T-n+1$. Hence for almost every $s$ it holds
    \begin{equation}\label{eq-growth:aux1}
        A(s)\leq e^{n-1}S(s).
    \end{equation}
    Next, the isoperimetric inequality and coarea formula provide
    \begin{equation}\label{eq-growth:aux2}
        A(s)+S(s)\geq c_IV(s)^{\frac{n-1}n},\qquad S(s)=-\frac{d}{ds}V(s),
    \end{equation}
    for almost every $s$. Combining \eqref{eq-growth:aux1} \eqref{eq-growth:aux2}, we have the differential inequality
    \begin{equation}\label{eq-growth:aux3}
        \frac d{ds}V(s)\leq -\frac{c_I}{1+e^{n-1}}V(s)^{\frac{n-1}n}.
    \end{equation}
    Integrating in $s\in\big[C_1e^{T/(n-1)-1},C_1e^{T/(n-1)}\big]$, and using \eqref{eq-growth:cond_T_2}, we have
    \begin{equation}\label{eq-growth:aux4}
        V\big(C_1e^{T/(n-1)-1}\big)^{\frac1n}\geq C_1e^{T/(n-1)}\cdot\frac{c_I(1-e^{-1})}{n(1+e^{n-1})}.
    \end{equation}
    On the other hand, by the isoperimetric inequality we have
    \begin{equation}\label{eq-growth:aux5}
        V\big(C_1e^{T/(n-1)-1}\big) \leq |E_T(u_\rho)|\leq c_I^{-\frac{n}{n-1}}\P{E_T(u_\rho)}^{\frac n{n-1}},
    \end{equation}
    and by \eqref{eq-growth:small_spheres} and the exponential growth of area, we have
    \begin{equation}\label{eq-growth:aux6}
        \P{E_T(u_\rho)}\leq e^{T-(n-1)\log\rho}\P{B_\rho(o)}\leq e^T\rho^{1-n}\cdot 2\omega_{n-1}\rho^{n-1}=2\omega_{n-1}e^T.
    \end{equation}
    Combining \eqref{eq-growth:aux4} \eqref{eq-growth:aux5} \eqref{eq-growth:aux6}, we obtain
    \[\frac{c_I(1-e^{-1})}{n(1+e^{n-1})}\cdot C_1e^{\frac T{n-1}}\leq c_I^{-\frac1{n-1}}(2\omega_{n-1})^{\frac1{n-1}}e^{\frac T{n-1}},\]
    contradicting \eqref{eq-growth:choice_of_C} (note that the factors $e^{T/(n-1)}$ cancel out). The claim is thus proved.\ \ $\Box$
    
    \vspace{3pt}
    
    We are ready to take the limit $\rho\to 0$. Since the solutions $u_\rho$ come from Huisken-Ilmanen's elliptic regularization, we have the gradient estimate \cite[Theorem 2.7]{Xu_2023_proper}
    \begin{equation}\label{eq-growth:gradient}
        |\D u_\rho|(x)\leq\frac{C(n)}{\min\big\{r(x),\sigma(x)\big\}},\qquad\forall\,\rho\leq r(x)/2,
    \end{equation}
    where $\sigma(x)$ depends only on the local geometry of $M$ near $x$. Combining \eqref{eq-growth:gradient} and Claim 1,2, we can use the Arzela-Ascoli theorem to extract a limit (for some subsequence)
    \[u=\lim_{\rho\to0}u_\rho\qquad\text{in}\ \ C^0_{\loc}\big(M\backslash\{o\}\big).\]
    It follows from Claim 1,2 that
    \begin{equation}
        \big|u-(n-1)\log r\big|<2r\qquad\text{on}\ \ B_{\rho_0/2}(o)\backslash\{o\},
    \end{equation}
    and it follows from \eqref{eq-growth:lower_bound} that
    \begin{equation}
        u\geq(n-1)\log r-(n-1)\log C_1\qquad\text{on}\ \ M\backslash\{o\}.
    \end{equation}
    Finally, by \cite[Theorem 2.1]{huiskenilmanen}, $u$ solves the IMCF in $M\backslash\{o\}$. The theorem then follows.
\end{proof}

\section{Approximation of IMCF with outer obstacle}

The goal of this section is to prove Theorem \ref{thm-intro:obstacle} regarding the convergence of $p$-harmonic functions with Dirichlet conditions. For the reader's convenience, we restate it here:

\begin{theorem}\label{thm-phar:main}
	Let $\Omega\Subset M$ be a smooth domain, and $E_0\Subset\Omega$ be a $C^{1,1}$ domain. Let $v_p$ be the unique solution of the equation
	\begin{equation}\label{eq-phar:vp}
		\left\{\begin{aligned}
			& \Delta_p v_p=0\qquad\text{in }\Omega\backslash\overline{E_0}, \\
			& v_p=1\qquad\text{on }\p E_0, \\
			& v_p=0\qquad\text{on }\p\Omega,
		\end{aligned}\right.
	\end{equation}
	and set $u_p=(1-p)\log v_p$. Then, as $p\to1$, $u_p$ converges in $C^0_{\loc}(\Omega)$ to the unique solution of IMCF in $\Omega$ with initial value $E_0$ and outer obstacle $\p\Omega$, as given by \cite[Theorem 1.6]{Xu_2024_obstacle}.
\end{theorem}

The proof makes essential use of Theorem \ref{thm-intro:BPP}. It is crucial for us that this theorem produces a family of $p$-IMCFs, rather than just a subsequence as $p \to 1$. To prove Theorem \ref{thm-phar:main}, we need to further recall the construction of IMCF with outer obstacle in \cite{Xu_2024_obstacle}. Let $\Omega,E_0$ be as in Theorem \ref{thm-phar:main}. Let $d(x)$ be the signed distance function from $\p\Omega$, such that it is negative in $\Omega$ and positive in $M\backslash\Omega$. For $\delta\in\RR$, denote the domain
\[\Omega_\delta=\big\{x\in M: d(x)<\delta\big\}.\]
In \cite[Subsection 6.1]{Xu_2024_obstacle}, a family of functions $\big\{\psi_\delta\in C^\infty(\Omega_\delta)\big\}_{0<\delta\ll1}$ is constructed so that:
\begin{enumerate}[nosep]
    \item $\psi_\delta\equiv0$ on $\Omega$,
    \item $\psi_\delta\to+\infty$ near $\p\Omega_\delta$,
    \item the new metric $g'_\delta=e^{2\psi_\delta/(n-1)}g$ on $\Omega_\delta$ is complete (note that $g'_\delta=g$ in $\Omega$).
\end{enumerate}
In \cite[Lemma 6.3]{Xu_2024_obstacle} the following was shown: for each $\delta\ll1$ and $\lambda\in(0,1)$, there is a constant $\epsilon(\delta,\lambda)$ so that for all $0<\epsilon\leq\epsilon(\delta,\lambda)$, there is a solution $u_{\epsilon,\delta,\lambda}$ of the following elliptic regularized equation
\begin{equation}\label{eq-phar:ellreg}
    \left\{\begin{aligned}
    & \operatorname{div}_{g'_\delta}\Big(\frac{\D_{g'_\delta}u_{\epsilon,\delta,\lambda}}{\sqrt{\epsilon^2+|\D_{g'_\delta}u_{\epsilon,\delta,\lambda}|^2}}\Big)=\sqrt{\epsilon^2+|\D_{g'_\delta}u_{\epsilon,\delta,\lambda}|^2}\qquad\text{on $\Omega_{\lambda\delta}\backslash\overline{E_0}$}, \\
    & u_{\epsilon,\delta,\lambda}=0\qquad\text{on $\p E_0$}.\\
    & u_{\epsilon,\delta,\lambda}=\underline u_\delta-2\qquad\text{on $\p\Omega_{\lambda\delta}$},
\end{aligned}\right.
\end{equation}
where $\underline u_\delta$ is defined in \cite[Lemma 6.2]{Xu_2024_obstacle}. We omit the detailed construction of $\psi_\delta$ and $\underline u_\delta$ here since it does not play any role in the proof below.

Moreover, in \cite[Lemma 6.4, Proposition 6.5\,--\,6.7]{Xu_2024_obstacle} the following is shown: for any sequences $\delta_i\to0$, $\lambda_i\to 1$, $\eps_i\leq\eps(\delta_i,\lambda_i)$ with $\epsilon_i\to0$, there exists a sequence of solutions $u_i=u_{\epsilon_i,\delta_i,\lambda_i}$ of \eqref{eq-phar:ellreg}, such that a subsequence of $u_i$ converges in $C^0_{\loc}(\Omega\backslash E_0)$ to the unique weak IMCF with initial value $E_0$ and outer obstacle $\p\Omega$.

\begin{proof}[Proof of Theorem \ref{thm-phar:main}] 
    It is sufficient to show that for any sequence $p\to1$, there is a subsequence so that the corresponding $u_p$ converges to the IMCF with outer obstacle (since the limit solution is a unique object, this shows that the entire family of $u_p$ also converges).

    The convergence of $u_p$ up to subsequence follows from the interior gradient estimate \cite{kotschwarni} and the Arzela-Ascoli theorem. Let $\tilde u=\lim_{j\to\infty}u_{p_j}$ be a subsequential limit, where $p_j\to1$ as $j\to\infty$. By the classical argument of Moser \cite{moser} (see also \cite{kotschwarni,ajm}), $\tilde u$ is a weak IMCF in $\Omega$ with initial value $E_0$. Then let $u$ the unique weak IMCF in $\Omega$ with initial value $E_0$ and obstacle $\p\Omega$, given by \cite[Theorem 6.1]{Xu_2024_obstacle}. Our goal is to show that $\tilde u=u$. By \cite[Theorem 6.1(iv)]{Xu_2024_obstacle}, namely the maximality of $u$, it suffices to show that $\tilde u\geq u$ in $\Omega\backslash E_0$.

    Fix a sequence $\delta_i\to0$. Recall the domain $\Omega_{\delta_i}$ and metric $g'_i=g'_{\delta_i}$ defined above. By \cite[Lemma 6.2]{Xu_2024_obstacle} (see also \cite[Lemma 2.10]{Xu_2024_obstacle}), the domain $(\Omega_{\delta_i},g'_i)$ admits a proper subsolution of IMCF. Therefore, there is a proper solution $w_i$ of IMCF in $(\Omega_{\delta_i},g'_i)$ with initial value $E_0$. Let $\lambda_i\in(0,1)$ and $\epsilon_i\ll1$, so that a solution $u_i=u_{\epsilon_i,\lambda_i,\delta_i}$ in \eqref{eq-phar:ellreg} exists. By Huisken-Ilmanen's elliptic regularization, we may choose $\lambda_i$ sufficiently close to 1 (depending on $i,\delta_i$) and then $\epsilon_i$ sufficiently small (depending on $i,\delta_i,\lambda_i$), such that
    \begin{equation}\label{eq-phar:aux1}
	\|u_i-w_i\|_{C^0(\Omega\backslash E_0)}\leq i^{-1}.
    \end{equation}
    By Theorem \ref{thm-intro:BPP} (applied to $\Omega$ in the new manifold $(\Omega_{\delta_i},g'_i)$), for each $i$ there exists a large enough index $j_i$ and a function $f_i\in\Lip(\bar\Omega\backslash E_0)$, such that
    \begin{equation}\label{eq-phar:aux2}
        f_i|_{\p E_0}=0,\qquad \exp\Big(\frac{f_i}{1-p_{j_i}}\Big)\ \ \text{is $p_{j_i}$-harmonic,}\qquad \|w_i-f_i\|_{C^0(\Omega\backslash E_0)}\leq i^{-1}.
    \end{equation}
    Note that the function $v_{p_{j_i}}$ in \eqref{eq-phar:vp} is the minimal $p_{j_i}$-harmonic function in $\Omega$ that takes value $1$ on $\p E_0$. Hence $v_{p_{j_i}}\leq\exp\big(f_i/(1-p_{j_i})\big)$, which implies
    \begin{equation}\label{eq-phar:aux3}
        u_{p_{j_i}}\geq f_i\qquad\text{in}\ \ \Omega\backslash E_0.
    \end{equation}
    Combining \eqref{eq-phar:aux1} \eqref{eq-phar:aux2} \eqref{eq-phar:aux3}, we obtain
    \[u_{p_{j_i}}\geq u_i-2i^{-1}\qquad\text{in}\ \ \Omega\backslash E_0.\]
    Since $\lim_{i\to\infty}u_{p_{j_i}}=\tilde u$ as well as $\lim_{i\to\infty}u_i=u$, it follows that $\tilde u\geq u$ in $\Omega\backslash E_0$.
\end{proof}

\appendix
\section{Asymptotic of the $p$-Green's kernel near the pole}\label{sec:green}

We prove Theorem \ref{teo_localsingular} by adapting the argument in \cite{kichenveron}, see also \cite[pp. 243-251]{veron}. First, by \cite{Serrin_2} it holds 
\begin{equation}\label{boundserrin}
   \gr \asymp \mu(r) \qquad \text{as } \, r(x) \to 0.
\end{equation}
Fix $R_0>0$ such that $B_{2R_0} \doteq B_{2R_0}(o) \subset \Omega$ and does not intersect $\cut(o)$, and let $\bar \kappa, \hat \kappa >0$ satisfy
$$
-\hat \kappa^2 \le \Sect \le \bar{\kappa}^2 \qquad \text{on } \, B_{2R_0},
$$
where $B_{2R_0}^*$ is the ball with the origin $o$ removed. Up to reducing $R_0$, we can suppose that $4R_0< \pi/\bar \kappa$. Let $\bar \gr, \hat \gr$ be the Green kernels of the geodesic ball $B_{2R_0}$ in a space form of curvature, respectively, $\bar \kappa^2$ and $- \hat \kappa^2$, transplanted to $M$. Take polar coordinates $(s,\theta) : \BB_{2R_0}^* \subset \R^n \ra B_{2R_0}^*$, and let
$$
g = \di s^2 + s^2 g_{\alpha\beta}(s,\theta) \di \theta^\alpha \di \theta^\beta
$$
be the corresponding expression of the metric $g$. For $R \in (0, R_0]$, consider the function
\begin{equation}\label{def_function}
\frac{\gr(x) - \|\gr\|_{L^\infty(\partial B_{R})}}{\bar \gr(x)}
\end{equation}
on $\BB_{R_0}^*$, and define
$$
\gamma : (0, R] \ra \R, \qquad \ol\gamma(s) \doteq \max \left\{ \frac{\gr(x) - \|\gr\|_{L^\infty(\partial B_{R})}}{\bar \gr(x)} \ : \ r(x) \in [s,R] \right\}.
$$
Since $\gr$ diverges at $o$, by the maximum principle $\ol\gamma$ is non-increasing, $\ol\gamma(R) \le 0$ while $\ol\gamma(s) >0$ for $s$ small enough, thus by comparison $\ol\gamma(s)$ is attained on $\partial B_s$. Indeed, if by contradiction this does not hold, setting
$$
C = \max_{\partial B_s} \left\{ \frac{\gr - \|\gr\|_{L^\infty(\partial B_R)}}{\bar \gr}\right\} >0
$$
for $\eps>0$ small enough the set
$$
U_\eps \doteq \left\{ x : r(x) \in B_R \backslash \overline{B}_s, \ \ \frac{\gr(x) - \|\gr\|_{L^\infty(\partial B_R)}}{\bar \gr(x)} > (1+\eps)C \right\}
$$
would be non-empty, and by the Laplacian comparison theorem
$$
\left\{
\begin{array}{l}
\disp \Delta_p  \big[\gr - \|\gr\|_{L^\infty(\partial B_R)}\big] = 0 \ge \Delta_p \big[(1+\eps)C\bar \gr\big] \qquad \text{on } \, U_\eps, \\[0.2cm]
\gr - \|\gr\|_{L^\infty(\partial B_R)} = (1+\eps)C\bar \gr \qquad \text{on } \, \partial U_\eps.
\end{array}
\right.
$$
by comparison, $\gr - \|\gr\|_{L^\infty(\partial B_R)} \le (1+\eps)C\bar \gr$ on $U_\eps$, contradiction.\par
Define
$$
\gamma^* = \limsup_{x \ra o} \frac{\gr(x)}{\bar \gr(r(x))} = \lim_{s \ra 0} \ol \gamma(s) \in (0,\infty).
$$
For $\lambda \in (0,1)$ define the dilation
\begin{equation}\label{def_Tlambda}
T_\lambda : \BB_{\frac{R_0}{\lambda}}^* \ra \BB_{R_0}^*, \qquad T_\lambda(t,\theta) = (\lambda t, \theta),
\end{equation}
and the rescaled metric $g_\lambda \doteq \lambda^{-2} T_\lambda^*g = \di t^2 + t^2g_{\alpha\beta}(\lambda t, \theta) \di \theta^\alpha \di \theta^\beta$ on $\BB_{R_0/\lambda}^*$. Note that, by homogeneity, $\gr_\lambda = \gr \circ T_\lambda$ solves
$$
\Delta_{p,\lambda} \gr_\lambda = 0 \qquad \text{on } \, \BB^*_{R_0/\lambda},
$$
where $\Delta_{p, \lambda}$ is the $p$-Laplacian with respect to $g_\lambda$. For a fixed $R \in (0, R_0)$ set
$$
u_\lambda(r,\theta) = \frac{\gr(\lambda r,\theta)}{\mu(\lambda R)}, \qquad \text{on } \, \BB^*_{R_0/\lambda},
$$
with $\mu$ the Green kernel on $\R^n$ as in \eqref{def_mup}. For $j \in \mathbb{N}$, $j \ge 2$ and each $\lambda$  small enough that $jR < R_0/\lambda$, because of \eqref{boundserrin} we then deduce
\begin{equation}\label{esti_udasopra}
|u_\lambda(r,\theta)| \le C_{j,p} \qquad \text{on } \, \BB_{jR} \backslash \BB_{\frac{R}{j}},
\end{equation}
for some constant $C_{j,p}$ depending on $j,p$ but not on $\lambda$. Since $g_\lambda$ converges to the Euclidean metric $g_0$ locally smoothly on $\R^m$, and since
\begin{equation}\label{equa_ulambda}
\Delta_{p,\lambda} u_\lambda = 0 \qquad \text{on } \, \BB^*_{R_0/\lambda},
\end{equation}
by elliptic estimates
\begin{equation}\label{elli_esti_C1beta}
\|u_\lambda\|_{C^{1,\beta_j}} \le \hat C_{j,p} \qquad \text{on } \, \BB_{jR} \backslash \BB_{\frac{R}{j}},
\end{equation}
for some $\beta_j \in (0,1)$. Therefore, for each fixed $j$ $\{u_\lambda\}$ has a convergent subsequence in $C^{1, \beta_j/2}(\BB_{jR} \backslash \BB_{\frac{R}{j}})$. By a diagonal argument, $u_\lambda$ subconverges in $C^1_\loc$ (and, on each fixed compact set, in an appropriate $C^{1,\beta}$) to some $u_0$ solving $\Delta_{p,0} u_0 = 0$ on $\R^n \backslash \{0\}$.\\[0.2cm]
\noindent \textbf{Step 1: } $\gr \sim \gamma^*\mu(r)$ as $r \ra 0$.\\[0.2cm]
Let $\{\lambda_k\}_{k \in \mathbb{N}}$, $\lambda_k \downarrow 0$ be a sequence such that $u_{\lambda_k} \ra u$ in $C^1_\loc(\R^n \backslash \{0\})$, and define $s_k = \lambda_k R$. Suppose that $k$ is large enough that $\ol \gamma(s_k)>0$. Let $x_k = (s_k, \theta_k) \in \partial \BB_{s_k}$ be a point where the function defined in \eqref{def_function} attains its maximum $\ol \gamma(s_k)$ over $\overline{B}_{R}\backslash B_{s_k}$. Write $x_k = T_{\lambda_k} y_k$, with $y_k = ( R, \theta_k)$. By compactness, $y_k \looparrowright y_0 = (R, \theta_0)$ as $k \ra \infty$. Letting $k \ra \infty$ along such subsequence and noting that $\bar \gr(s_k) \sim \mu(s_k)$ for $p \le n$ we deduce
\begin{equation}\label{first}
\begin{array}{lcl}
u_0(R, \theta_0) & = & \disp \lim_{k \ra \infty} u_0(R, \theta_k) = \lim_{k \ra \infty} u_{\lambda_k}(R, \theta_k) \\[0.4cm]
& = & \disp \lim_{k \ra \infty} \frac{\gr(s_k, \theta_k)}{\mu(s_k)} = \lim_{k \ra \infty} \ol \gamma(s_k) = \gamma^*.
\end{array}
\end{equation}
On the other hand, for each $(r,\theta)$ fixed and for $k$ large enough that $\lambda_k < R/r$,
\begin{equation}\label{second}
\begin{array}{lcl}
u_0(r, \theta) & = & \disp \lim_{k \ra \infty} u_{\lambda_k}(r, \theta) = \lim_{k \ra \infty} \frac{\gr(\lambda_k r, \theta)}{\mu(\lambda_k R)} \\[0.4cm]
& = & \disp \lim_{k \ra \infty} \frac{\gr(s_k r/R, \theta)}{\mu(s_k)} = \lim_{k \ra \infty} \frac{\gr(s_k r/R, \theta)}{\bar \gr(s_kr/R)} \cdot \frac{\mu(s_kr/R)}{\mu(s_k)} \\[0.5cm]
& \le & \disp \gamma^* \lim_{k \ra \infty} \frac{\mu(s_kr/R)}{\mu(s_k)} = \left\{ \begin{array}{ll} \gamma^* \frac{\mu(r)}{\mu(R)} & \quad \text{if } \, p<n \\[0.3cm]
\gamma^* & \quad \text{if } \, p=n.
\end{array}\right.
\end{array}
\end{equation}
Therefore, $\Delta_{p,0} u_0 =0$ on $\R^n \backslash \{0\}$ and
$$
\begin{array}{llll}
\text{if $p<n$,} & \disp \quad u_0(r,\theta) \le \gamma^* \frac{\mu(r)}{\mu(R)} \qquad \text{on } \, \R^n \backslash \{0\}, & \quad u_0(r,\theta) = \gamma^* \frac{\mu(r)}{\mu(R)} \quad \text{at $y_0$}, \\[0.4cm]
\text{if $p=n$,} & \disp \quad u_0(r,\theta) \le \gamma^* \qquad \text{on } \, \R^n \backslash \{0\}, & \quad u_0(r,\theta) = \gamma^* \quad \text{at $y_0$}.
\end{array}
$$
By the strong maximum principle (\cite[Prop. 3.3.2]{tolksdorf2}; a version directly applicable to the manifold setting can be found in \cite[Thm. 2.5.2]{pucciserrin} and \cite[Thm. 1.2]{puccirigoliserrin}),
\begin{equation}\label{coseuo}
u_0(r,\theta) = \left\{ \begin{array}{ll}
\gamma^* \mu(r)/\mu(R) & \quad \text{if } \, p <n \\[0.3cm]
\gamma^* & \quad \text{if } \, p=n.
\end{array}\right.
\end{equation}
The uniqueness of any subsequential limit shows that the entire family $\{u_\lambda\}$ locally converges in $C^1$ to $u_0$ as $\lambda \ra 0$. Consequently, if $p<n$
$$
\disp \frac{\gamma^*}{\mu(R)} =  \disp \lim_{\lambda \ra 0} \frac{u_\lambda(r,\theta)}{\mu(r)} = \lim_{\lambda \ra 0} \frac{\gr(\lambda r,\theta)}{\mu(\lambda R)\mu(r)} = \lim_{\lambda \ra 0} \frac{\gr(\lambda r,\theta)}{\mu(\lambda r)\mu(R)} = \frac{1}{\mu(R)} \lim_{s \ra 0} \frac{\gr(s,\theta)}{\mu(s)},
$$
and if $p=n$
$$
\disp \gamma^* = \disp \lim_{\lambda \ra 0} u_\lambda(r,\theta) = \lim_{\lambda \ra 0} \frac{\gr(\lambda r,\theta)}{\mu(\lambda R)} = \lim_{\lambda \ra 0} \frac{\gr(\lambda r,\theta)}{\mu(\lambda r)} = \lim_{s \ra 0} \frac{\gr(s,\theta)}{\mu(s)}.
$$
Hence, $\gr(s,\theta) \sim \gamma^*\mu(s)$ as $s \ra 0$, as claimed.\\[0.2cm]
\noindent \textbf{Step 2:} it holds
$$
\gamma^* \hat \gr \le \gr \le \|\gr\|_{L^\infty(\partial B_R)} + \gamma^*\bar \gr \qquad \text{on } \, B_R^*(o).
$$
\emph{Proof:} Because of Step 1, for $\eps>0$ we can apply the comparison theorem to $\gr$, $(\gamma^*-\eps)\hat \gr$ and $(\gamma^*+\eps)\bar \gr$ to deduce
$$
(\gamma^*-\eps)\hat \gr \le \gr \le \|\gr\|_{L^\infty(\partial B_R)} + (\gamma+\eps)\bar \gr,
$$
and let $\eps \ra 0$.\\[0.2cm]
\noindent \textbf{Step 3:} it holds
$$
|\nabla \gr - \gamma^* \mu'(r) \nabla r| = o\Big(|\mu'(r)|\Big) \qquad \text{as } \, r \ra 0.
$$
\emph{Proof: } From the convergence $u_\lambda \ra u_0$ in $C^1_\loc$, we deduce
\begin{equation}\label{convegradiente}
\| \bar \nabla u_\lambda - \bar \nabla u_0 \|_{L^\infty( \BB_{jR} \backslash \BB_{R/j})} = o_j(1)
\end{equation}
as $\lambda \ra 0$, where $o_j(1)$ is a quantity depending on $j$ which vanishes as $\lambda \ra 0$, and $\bar \nabla$ denotes the Euclidean gradient.\\[0.2cm]
\noindent \textbf{The case $p<n$}.\\
From \eqref{convegradiente} we get
\begin{equation}\label{normzero}
\left[\partial_r u_\lambda - \gamma^* \frac{\partial_r\mu}{\mu(R)}\right]^2 + \frac{1}{r^2}g^{\alpha\beta}(0,\theta)(\partial_\alpha u_\lambda)(\partial_\beta u_\lambda) \ra 0
\end{equation}
locally uniformly on $(\R^n)^*$. Now, for $r \in [R/j, jR]$ we get
$$
\begin{array}{l}
\disp \frac{1}{r^2}\left|g^{\alpha\beta}(0,\theta)(\partial_\alpha u_\lambda)(\partial_\beta u_\lambda) - g^{\alpha\beta}(\lambda r,\theta)(\partial_\alpha u_\lambda)(\partial_\beta u_\lambda)\right| \\[0.5cm]
\qquad \le \disp \frac{j^2}{R^2}\left|g^{\alpha\beta}(\lambda r,\theta)- g^{\alpha\beta}(0,\theta)\right| \left|(\partial_\alpha u_\lambda)(\partial_\beta u_\lambda)\right|
\end{array}
$$
and, since $\partial_\alpha u_0 = 0$, the right-hand side goes to zero uniformly on $\BB_{jR}\backslash \BB_{R/j}$ as $\lambda \ra 0$. It follows from \eqref{normzero} that
\begin{equation}\label{normzero_2}
\partial_r u_\lambda - \gamma^* \frac{\partial_r \mu}{\mu(R)} \ra 0, \qquad \frac{1}{r^2}g^{\alpha\beta}(\lambda r,\theta)(\partial_\alpha u_\lambda)(\partial_\beta u_\lambda) \ra 0
\end{equation}
locally uniformly on $\BB_{jR} \backslash \BB_{R/j}$, as $\lambda \ra 0$. Setting $s = \lambda r$, and using the identity
\begin{equation}\label{ide_mu}
(\partial_s \mu)(s) = - \frac{ n-p}{p-1} \frac{\mu(s)}{s},
\end{equation}
we infer
$$
\partial_r u_\lambda  = \disp \frac{(\partial_s \gr)(\lambda r ,\theta)}{\mu(\lambda R)} \lambda =  \frac{(\partial_s \gr)(\lambda r ,\theta)}{(\partial_s \mu)(\lambda r)}\left[ - \frac{n-p}{p-1} \frac{\mu(r)}{r\mu(R)}\right].
$$
Therefore, from the first in \eqref{normzero_2} we deduce
$$
0 \leftarrow \frac{(\partial_s \gr)(\lambda r ,\theta)}{(\partial_s \mu)(\lambda r,\theta)}\left[ - \frac{n-p}{p-1} \frac{\mu(r)}{r}\right] - \gamma^* \partial_r \mu,
$$
as $\lambda \ra 0$, uniformly on $\BB_{jR}\backslash \BB_{R/j}$, that is, substituting $\lambda r =s$,
\begin{equation}\label{radial_pmm}
0 = \lim_{s \ra 0} \frac{(\partial_s \gr)(s,\theta)}{(\partial_s \mu)(s)} + \gamma^* \frac{r \partial_r \mu}{\mu(r)} \frac{p-1}{n-p} = \lim_{s \ra 0} \frac{(\partial_s \gr)(s,\theta)}{(\partial_s \mu)(s)} - \gamma^*
\end{equation}
as $s \ra 0$. Analogously, from the second in \eqref{normzero_2} we obtain as $\lambda \ra 0$
$$
0 \leftarrow \disp \frac{1}{r^2}g^{\alpha\beta}(s,\theta)(\partial_\alpha u_\lambda)(\partial_\beta u_\lambda) = \frac{1}{r^2 \mu^2(\lambda r)} \frac{\mu^2(r)}{\mu^2(R)} g^{\alpha\beta}(s,\theta)(\partial_\alpha \gr)(\partial_\beta \gr),
$$
and using \eqref{ide_mu},
$$
0 = \lim_{s \ra 0} \frac{1}{(\partial_s \mu)^2(s)} \frac{1}{s^2} g^{\alpha\beta}(s,\theta)(\partial_\alpha \gr)(\partial_\beta \gr).
$$
This, together with \eqref{radial_pmm} proves Step 3.\\[0.2cm]
\noindent \textbf{The case $p=n$.}\\[0.2cm]
Computations are analogous to those for $p<n$, but we need more refined gradient estimates to conclude. See \cite[pp.248-251]{veron} for a complete proof, and we leave to the reader its implementation in a manifold setting.\\[0.2cm]

\noindent \textbf{Step 4:} $\gamma^*=1$.\\
Fix $\eps$ such that $B_{\eps}(o) \Subset \Omega$. We use the weak definition of $\Delta_p \gr = -\delta_o$ with a radial test function $\phi \in C^\infty_c(B_{\eps}(o))$ such that $\phi(o)=1$ to obtain
$$
1 = \phi(o) = \disp \int_\Omega |\nabla \gr|^{p-2} \langle \nabla \gr, \nabla \phi \rangle = \int_{B_\eps} |\nabla \gr|^{p-2} (\partial_r \gr)(\partial_r\phi).
$$
From Step 4, we obtain
$$
\begin{array}{l}
\disp 1 = \int_{B_\eps} |\nabla \gr|^{p-2} (\partial_r \gr)(\partial_r\phi) \\[0.5cm]
\qquad \qquad = \disp \big(1+o_\eps(1)\big)\int_0^\eps (\gamma^*)^{p-1} |\mu'(r)|^{p-2} \mu'(r) \phi'(r) \omega_{m-1} r^{m-1} \di r = (\gamma^*)^{p-1}(1+o_\eps(1)),
\end{array}
$$
and the thesis follows by letting $\eps \ra 0$.\\[0.2cm]

\noindent \textbf{Step 5:} it holds
$$
\left| \nabla^2 \gr - \nabla^2 \mu(r) \right| = o \Big( |\mu''(r)| \Big) \qquad \text{as } \, r \ra 0.
$$
Because of Step 3, having fixed $j \in \mathbb{N}$ choose $\lambda_j$ small enough that
\begin{equation}\label{gradi_vicinozero}
\left| \frac{\nabla \gr(s,\theta)}{\mu'(s)} \right| \in \left[ \frac{1}{2}, 2\right] \qquad \text{on } \, B^*_{R\lambda_j/j}.
\end{equation}
Rearranging \eqref{equa_ulambda}, by homogeneity $u_\lambda$ also solves the linear PDE
$$
\diver_{g_\lambda} \left( \left| \frac{\nabla \gr_\lambda}{\mu'(\lambda R)} \right|^{p-2} \nabla u_\lambda \right) = 0,
$$
that in view of \eqref{gradi_vicinozero} is uniformly elliptic on $\BB^*_{jR}\backslash \BB_{j/R}^*$ for every $\lambda < \lambda_j$, with uniform $C^{0, \beta}$ estimates on its coefficients because of \eqref{elli_esti_C1beta}. Via Schauder's estimates we therefore improve the local $C^{1,\beta}$ estimates to
$$
\|u_\lambda\|_{C^{2,\beta_j}} \le \hat C_{j,p,R} \qquad \text{on } \, \BB_{jR} \backslash \BB_{\frac{R}{j}},
$$
for some $\beta_j \in (0,1)$ and a constant now also depending on $R$, The proof now follows as in Step 3, and is left to the reader.

\addcontentsline{toc}{section}{References}

\bibliographystyle{amsplain}

\end{document}